\documentclass[11pt]{article}

\usepackage{xcolor}

\definecolor{refkey}{gray}{.8}
\definecolor{labelkey}{rgb}{0.4,0.4,0.8}

\usepackage{amsmath,amsthm,a4,latexsym,amssymb}
\usepackage[unicode]{hyperref}

\hypersetup{
    colorlinks=true,
    linktoc=all,
    linkcolor=blue,
}

\usepackage{srcltx}
\usepackage{geometry,mathrsfs,enumitem}
\usepackage{amstext,bbm,nicefrac,physics,tikz}
\usepackage{subcaption}
\usepackage{todonotes}
\usetikzlibrary{calc}
\usetikzlibrary{decorations.pathmorphing}

\makeatletter
\@addtoreset{equation}{section}
\makeatother

\newfont{\bfit}{cmbxti10 scaled 1200}

\renewcommand{\d}{{\rm d}}

\newcommand{\supp}{\mathop{\mathrm{supp}}}
\newcommand{\dist}{\mathop{\mathrm{dist}}}
\newcommand{\e}{\mathrm e}
\newcommand{\R}{\mathbb R}
\newcommand{\N}{\mathbb N}
\newcommand{\Z}{\mathbb Z}

\newcommand{\eps}{\varepsilon}

\newcommand{\E}{\mathbb E}

\renewcommand{\P}{\mathbb P}

\def\1{{\mathchoice
{1\mskip-4mu\mathrm l}
{1\mskip-4mu\mathrm l}
{1\mskip-4.5mu\mathrm l}
{1\mskip-5mu\mathrm l}}}

\newcommand{\heap}[2]{\genfrac{}{}{0pt}{}{#1}{#2}}

\newcommand{\ssup}[1]{{\scriptscriptstyle({#1})}}

\newenvironment{example}
{
\refstepcounter{theorem}
{\bf Example \thetheorem\ }\nopagebreak
}
{
\nopagebreak {\hfill\rule{2mm}{2mm}}\\
}

\newtheorem{theorem}{Theorem}[section]
\newtheorem{lemma}[theorem]{Lemma}
\newtheorem{cor}[theorem]{Corollary}

\newtheorem{definition}[theorem]{Definition}

\newtheoremstyle{thm}
{1.5ex}
{1.5ex}
{\itshape\rmfamily}
{}
{\bfseries\rmfamily}
{}
{2ex}
{}

\newtheoremstyle{rem}
{1.3ex}
{1.3ex}
{\rmfamily}
{}
{\itshape\rmfamily}
{}
{1.5ex}
{}

\theoremstyle{rem}

\newtheorem{remark}{Remark}[section]

\begin{document}

\title{\textsc{Degenerate Diffusions on Continuum Percolation and Hamilton-Jacobi-Bellman Equations}}

\maketitle

 




\begin{center}
{\sc Rodrigo Bazaes}\footnote{rodrigobazaes.com, {\tt rodrigo@rodrigobazaes.com}},
{\sc Alexander Mielke}\footnote{WIAS Berlin and Humboldt Univerit\"at Berlin, Mohrenstra\ss{}e 39,  10117 Berlin,
    {\tt alexander.mielke@wias-berlin.de}}  
and  
{\sc Chiranjib Mukherjee}\footnote{Universit\"at M\"unster, Einsteinstrasse 62,
      48149 M\"unster, Germany, {\tt chiranjib.mukherjee@uni-muenster.de}}
\\[0.5em]
\textit{Universit\"at M\"unster, WIAS Berlin and HU Berlin, Universit\"at
  M\"unster}
\\[0.5em]
\today
\end{center}

{
\renewcommand{\thefootnote}{}
\footnote{\textit{AMS Subject
Classification:} {35F21, 35B27, 49L25, 60F10, 78A48, 82B43}}
\footnote{\textit{Keywords:} Degenerate diffusions, continuum percolation, controlled diffusions, Hamilton-Jacobi-Bellman equations, stochastic homogenization,}
}

\begin{quote}{\small {Abstract: We study degenerate controlled diffusions and
Hamilton--Jacobi--Bellman equations posed on genuine continuum percolation
clusters. The
diffusion is constrained to evolve inside the infinite cluster and degenerates
according to the distance to the irregular random boundary. Under
suitable regularity and degeneracy assumptions on the diffusion matrix, we
prove that the corresponding controlled diffusion never reaches the boundary
and therefore admits a unique global strong solution.

Using this result, we establish a stochastic control representation formula
for viscosity solutions of degenerate Hamilton--Jacobi--Bellman equations
posed on the random cluster, without imposing boundary conditions.

We further verify that the structural assumptions introduced in this work,
including quantitative integrability properties of the
distance-to-the-boundary function and the associated degeneracy, hold for
concrete continuum percolation models such as the Boolean model. In
particular, although the diffusion never reaches the boundary, the boundary
geometry remains a fundamental ingredient of the theory through the admissible
degeneracy regime. The results establish a quenched framework for stochastic
control and degenerate partial differential equations on genuine continuum
percolation geometries, and identify a link between the analytic structure of
the diffusion and the stochastic geometry of the underlying cluster. They also
provide the foundational framework for the homogenization theory developed in
the companion article~\cite{BMM26}.}
}

\end{quote}


\tableofcontents

\section{Introduction}\label{sec-intro}

\subsection{Background}\label{subsec background}

Continuum percolation provides a natural class of random media in which the underlying geometry 
plays an important role. In the supercritical regime, the existence of an unbounded connected 
component $\mathcal C_\infty$ allows one to formulate stochastic processes and partial 
differential equations on a random domain. In contrast to the classical setting of 
diffusions in $\R^d$, however, the geometry of $\mathcal C_\infty$ introduces substantial 
additional difficulties: the domain is unbounded, the boundary $\partial\mathcal C_\infty$ 
may be highly irregular, and conditioning on the event $\{0\in\mathcal C_\infty\}$ destroys 
the inherent translation invariance of the underlying medium. These difficulties become 
particularly pronounced for degenerate diffusions and stochastic control problems, where both 
the geometry of the cluster and the behavior of the diffusion near the boundary play a 
fundamental role.

In this work we study stochastic control problems and Hamilton--Jacobi--Bellman equations on 
continuum percolation clusters under a class of degenerate diffusion coefficients. More precisely, 
we consider diffusion matrices $a(x,\omega)$ which are symmetric, nonnegative definite, 
sufficiently regular, and satisfy
\[
a(x,\omega)=0 \qquad\mbox{for } x\notin \mathcal C_\infty(\omega).
\]
Thus the associated diffusion is constrained to evolve inside the random domain and degenerates at and near the boundary of the cluster.

A central issue in this setting is the interplay between the regularity of the diffusion coefficient and the strength of its degeneracy. On the one hand, sufficient regularity of $a$ is required in order to ensure that the associated diffusion process remains in the interior of the cluster and does not hit the highly irregular boundary $\partial\mathcal C_\infty$. On the other hand, the large-scale analysis and homogenization theory  requires quantitative integrability properties on the degeneracy of $a$, expressed through a negative moment condition (see the discussion below Theorem \ref{theorem 3'}). These two requirements are, in principle, competing: stronger regularity restricts the admissible degeneracy, while stronger degeneracy worsens the relevant integrability properties. In continuum percolation, this tension is particularly delicate because the distance to the boundary may itself exhibit highly irregular behavior under the conditioned measure.

The first goal of the present paper is therefore to construct, under suitable regularity and degeneracy assumptions on the diffusivity $a$ and for almost every realization of the environment, a controlled quenched diffusion process evolving in $\mathcal C_\infty$ and show that it remains in the interior of the cluster for all times almost surely. As a consequence, the associated stochastic differential equation admits a unique global strong solution, cf. Theorem \ref{theorem 1'}. 

Our second main result uses this diffusion process to establish a stochastic control representation formula for viscosity solutions of degenerate Hamilton--Jacobi--Bellman equations posed on the random cluster, without requiring any boundary condition, cf. Theorem \ref{theorem 2'}.

The third main result of this article provides concrete examples of continuum percolation models for
which the structural assumptions introduced in this work can be verified. In particular, we establish
the required geometric and probabilistic estimates for the Boolean model, including the negative
moment bounds on the distance to the boundary needed for quantitative large-scale analysis in random
media. We also verify several of the structural assumptions for the continuum random cluster model,
thereby illustrating that the framework extends beyond finite-range dependent settings and includes
examples with long-range correlations, cf. Theorem \ref{theorem 3'}.

The verification theorem plays a particularly important role because the
assumptions imposed on the diffusion coefficient combine regularity,
degeneracy, and quantitative integrability properties which, in continuum
percolation, are all constrained by the geometry of the
distance-to-the-boundary function under the conditioned measure. In
particular, although the diffusion never reaches the boundary and no boundary
condition is imposed, the irregular boundary geometry nevertheless enters the
theory in a fundamental way through the admissible degeneracy regime. This
interplay becomes even more pronounced in the companion homogenization
theory~\cite{BMM26}, where the geometry of the cluster ultimately determines
the range of admissible degeneracies compatible with the coercivity of the
Hamiltonian.

Beyond its role in the homogenization theory developed in the companion work
\cite{BMM26}, the interest in the present work is also intrinsic. To the best
of our knowledge, a framework for degenerate controlled diffusions
and viscosity solutions for Hamilton--Jacobi--Bellman equations on genuine
continuum percolation clusters has not previously been available in the
literature. We now describe the mathematical setting more precisely.

\subsection{Outline of the results.}\label{subsec outline}

Let $\Omega$ denote the space of locally finite point configurations in $\R^d$, $d\ge2$, equipped with a probability measure $\P$. The translation group $\{\tau_x\}_{x\in\R^d}$ acts on $\Omega$ by
\[
\tau_x\omega=\omega-x:=\{y-x:y\in\omega\}.
\]
Given $\omega\in\Omega$, define
\[
\mathcal C(\omega):=\bigcup_{y\in\omega} B_{1/2}(y).
\]
We assume that $\P$ is stationary and ergodic under $\{\tau_x\}_{x\in\R^d}$ and that, $\P$-almost surely, $\mathcal C(\omega)$ contains a unique infinite unbounded connected component, denoted by $\mathcal C_\infty(\omega)$. The event
\[
\Omega_0:=\{0\in\mathcal C_\infty\}
\]
then has strictly positive probability, and we define the conditioned measure
\[
\P_0(\cdot):=\P(\cdot\mid \Omega_0).
\]
Unlike $\P$, the conditioned measure $\P_0$ is not invariant under translations. Next, for 
$\omega\in\Omega_0$ and $x\in\mathcal C_\infty(\omega)$, we consider the controlled diffusion
\begin{equation}\label{diff}
X_t=X_t^\omega
= x + \int_0^t \sigma(X_s,\omega)\,dB_s
+ \int_0^t (\mathrm{div}a)(X_s,\omega)\,ds
+ \int_0^t a(X_s,\omega)\cdot c(s)\,ds .
\end{equation}
Here $(B_t)_{t\ge0}$ is a Brownian motion on an auxiliary probability space $(\mathscr X, 
\mathcal F,P)$ independent of the point process, and $c:[0,T]\times\mathscr X\to\R^d$ denotes 
a progressively measurable square-integrable control.

The diffusion matrix $a(\omega)$ is assumed to be symmetric, positive semi-definite, and 
degenerately elliptic. More precisely, there exists a measurable map $\xi:\Omega_0\to(0,\infty)$ 
such that
\begin{equation}\label{a deg}
\begin{aligned}
\xi(\omega)|v|^2
\le \langle a(\omega)v,v\rangle
\lesssim |v|^2,
\qquad \forall v\in\R^d, \qquad\mbox{and}\qquad
\E_0[\xi^{-\chi}]<\infty,
\end{aligned}
\end{equation}
for a suitable $\chi\in(0,1)$; see Remark \ref{remark boundary} and the discussion below Theorem \ref{theorem 3'}. For $x\in\R^d$ we define
\[
a(x,\omega):=a(\tau_x\omega),
\]
so that $\{a(x,\cdot)\}_{x\in\R^d}$ is stationary with respect to $\P$ and $x\mapsto a(x,\omega)$ 
is sufficiently regular on the open set $ \mathcal C_\infty(\omega)$ for any fixed $\omega$. 

For $\omega\in\Omega_0$ and a square-integrable progressively measurable control $c$, let 
$P_x^{c,\omega}$ denote the law of the diffusion \eqref{diff}. Our first main result establishes 
that, under suitable assumptions on the regularity and degeneracy of $a$, the diffusion process 
never reaches the boundary of the cluster.

\begin{theorem}[Non-attainment of the boundary]\label{theorem 1'}
Under the stated assumptions \ref{assump:est-erg}-\ref{assump:inf-comp} (see Section \ref{subsec assump perc}) and \ref{f1} (see Section \ref{sec-main-results}), for every fixed $\omega\in\Omega_0$, every admissible control $c$, and every initial condition $x\in\mathcal C_\infty(\omega)$,
\[
X_t\in\mathcal C_\infty(\omega)
\qquad \text{for all } t\ge0
\qquad P_x^{c,\omega}\text{-almost surely}.
\]
Consequently, the stochastic differential equation \eqref{diff} admits a unique global strong solution. See Theorem \ref{theorem 1}.
\end{theorem}

We refer to Lemma~\ref{lemma} and Corollary~\ref{cor example a} for an explicit example of a diffusion matrix a satisfying the assumptions of the theorem, and to Remark~\ref{remark:sharpness} for a discussion of their significance. Theorem \ref{theorem 1'} is quite important for the stochastic control problem. Because trajectories remain in the interior of the cluster, no boundary condition is required for the associated Hamilton--Jacobi--Bellman equation.

For a uniformly continuous function $f:\R^d\to\R$, define
\[
u_\eps(t,x,\omega)
=
\sup_c \eps\,J^\omega_{f_\eps}
\left(\frac t\eps,\frac x\eps,c\right),
\]
where
\[
J_f^\omega(t,x,c)
=
E^{P_x^{c,\omega}}
\left[
f(X_t)-\int_0^t L(X_s,c(s),\omega)\,ds
\right].
\]
with $L(q,\omega)=\sup_{p\in\R^d}[\langle p,q\rangle_a- H(p,\omega)]$ being the Lagrangian corresponding to a convex Hamiltonian $p\mapsto H(p,\omega)$ satisfying a suitable growth condition $H(p,\omega) \equiv \|p\|_a^\alpha$ with $\alpha>1+\delta$  -- we refer to Section \ref{sec-HJB} for precise definitions and assumptions. Our second main result shows that this value function yields a viscosity solution of the Hamilton--Jacobi--Bellman equation posed on the random cluster.

\begin{theorem}[Control representation]\label{theorem 2'}
Under the stated assumptions \ref{assump:est-erg}-\ref{assump:inf-comp}, \ref{f1}, and \ref{f2}-\ref{f4} (see Section \ref{sec-HJB}), the value function $u_\eps$ is a viscosity solution of the Hamilton--Jacobi--Bellman equation \eqref{eq-HJB'} on
\[
(0,\infty)\times\mathcal C_\infty(\omega),
\]
in the class of viscosity solutions of at most linear growth, without requiring any boundary condition. See Theorem \ref{theorem 2}.
\end{theorem}

Finally, we verify that the structural assumptions imposed on the geometry of the cluster and on the degeneracy of $a$ hold for concrete continuum percolation models.

\begin{theorem}[Verification for continuum percolation models]\label{theorem 3'}
The assumptions \ref{assump:est-erg}-\ref{assump:dist} stated in Section \ref{subsec assump perc} hold for the continuum Boolean percolation model under suitable parameters. In addition, the structural assumptions \ref{assump:est-erg}-\ref{assump:inf-comp} are verified for the continuum random cluster model. See Theorem \ref{theorem 3}.
\end{theorem}

We point out that the negative moment condition $\E_0[\xi^{-\chi}]<\infty$ in \eqref{a deg} above is not required for the well-posedness results established in the present paper. Rather, it is necessary for our companion work in \cite{BMM26}. In this context, and as previously mentioned, the verification theorem above plays a particularly important role in view of the competing nature of the assumptions imposed on the diffusion coefficient. Indeed, sufficient regularity of $a$ is needed to ensure that the diffusion does not reach the highly irregular boundary $\partial\mathcal C_\infty$, whereas the homogenization theory needs the quantitative integrability of the degeneracy. In continuum percolation these conditions are intertwined through the geometry of the distance-to-the-boundary function under the conditioned measure $\P_0$. A natural example is given by
\[
\xi(\omega)\sim d(0,\partial\mathcal C_\infty(\omega)),
\]
or more precisely by a suitable regularized distance function; see Section \ref{sec exmaple a}. Theorem \ref{theorem 3'} shows that these competing requirements are simultaneously realizable in concrete continuum percolation models.


\subsection{Continuum percolation and assumptions}\label{sec-assume-models}
In this section we will provide the mathematical layout for the continuum percolation model and the related assumptions.

\subsubsection{Point processes.}\label{subsec pointPalm}

Fix an integer $d\geq 2$, and let $\Omega$ be the space of all locally finite
point subsets of $\R^d$. We denote by $\mathcal{B}(\R^d)$ the Borel
$\sigma$-algebra on $\R^d$. The Lebesgue measure will be denoted by $\lambda$
(or by $\lambda_d$ when we need to emphasize the dimension). We endow $\Omega$
with the smallest $\sigma$-algebra $\mathcal{G}$ that makes the maps
$\omega\mapsto \#(\omega\cap A)$ measurable for all
$A\in \mathcal{B}(\R^d)$, where $\#(\omega\cap A)$ denotes the cardinality of
$\omega\cap A$. A \textit{point process} is a probability measure $\P$ on
$(\Omega,\mathcal{G})$.

On $\Omega$, the group of translations $(\tau_x)_{x\in \R^d}$ acts naturally as
\[
\tau_x \omega:=\omega-x = \big\{ y-x  : y\in \omega \big\}.
\]
We say that a point process is {\it stationary} if
\begin{equation}\label{def-stationary}
\P\circ \tau_x = \P\qquad\forall x\in \R^d.
\end{equation}
A stationary point process is ergodic with respect to $(\tau_x)_{x\in \R^d}$
if
\begin{equation}
  \label{eq:erg-def}
 \forall\, A\in \mathcal{G}\  \forall \, x\in
     \R^d : \qquad \tau_x A= A\quad \Longrightarrow \quad \P(A)\in \{0,1\}.
\end{equation}
The \textit{intensity measure} of $\P$ is the measure on
$(\R^d,\mathcal{B}(\R^d))$ given by
\begin{equation}\label{eq:Theta-def}
 \Theta(A):= \int \#(\omega \cap A) \P(\d\omega)= \E[\#(\omega\cap A)].
\end{equation}
Here and throughout the sequel, $\E$ will denote expectation with respect to
$\P$. Notice that when $\P$ is stationary and $\Theta$ is locally finite, there
exists $\zeta\in (0,\infty)$ such that $\Theta=\zeta \lambda$. We call $\zeta$
the \textit{intensity} of the point process.

\subsubsection{Palm measures.}

On an intuitive level, {\it Palm measures} formalize the idea of the
distribution of a point process conditioned on containing a fixed point
$x\in \R^d$. First, we define the measure $\mathfrak{C}$ on
$\R^d\times \Omega$ by
\begin{equation}\label{eq:camp-meas-def}
  \mathfrak{C}(A):=\E\bigg[\sum_{x\in
    \omega}\1_A(x,\tau_x\omega)\bigg],
  \qquad
  A\in \mathcal{B}(\R^d)\otimes \mathcal{G}.
\end{equation}
The measure $\mathfrak{C}$ can be decomposed when $\P$ is stationary. Indeed,
by \cite[Theorem 3.3.1]{SW08}, if $\P$ is a stationary point process with
intensity $\zeta\in (0,\infty)$, then there exists a unique measure
$\P^{\ssup 0}$ on $(\Omega,\mathcal{G})$ such that
\begin{equation}\label{eq:palm-decomp}
\mathfrak{C}=\zeta \lambda\otimes \P^{\ssup 0}.
\end{equation}
We call $\P^{\ssup 0}$ the \textit{Palm measure} corresponding to $\P$. It can
be seen as the distribution of a point process conditioned on containing the
origin; see \cite[Proposition 9.5]{LP18}. In particular,
$\P^{\ssup 0}(0\notin \omega)=0$; see \cite[Eq. (9.7)]{LP18}. More generally,
we define
\[
\P^{\ssup x}:=\P^{\ssup 0}\circ \tau_x\qquad
\mbox{for } x\in \R^d.
\]
The aforementioned decomposition allows us to disintegrate $\P$ in terms of
$(\P^{\ssup x})_{x\in \R^d}$. Indeed, by \cite[Theorem 3.3.3]{SW08}, if $\P$ is
a stationary point process with intensity $\zeta\in (0,\infty)$, then for all
$f\in L^1(\R^d\times \Omega)$, the map
$\omega\mapsto \sum_{x\in \omega}f(x,\omega)$ is measurable, and
\begin{equation}\label{eq:refined-camp-decomp}
  \E\bigg[ \sum_{x\in \omega}f(x,\omega)\bigg]
  =
  \zeta\int_{\R^d}\E^{\ssup 0}[f(x,\tau_{-x}\omega)]\,\d x
  =
  \zeta \int_{\R^d}\E^{\ssup x}[f(x,\omega)]\,\d x.
\end{equation}
Similarly, one can define the $n$-fold Palm distribution
$\P^{\ssup{x_1,\cdots,x_n}}$ for $x_1,\cdots, x_n\in \R^d$. In this case, we
have the equality
\begin{equation}\label{eq:multidim-campbell-eq}
  \E\bigg[
  \sum_{\heap{x_1,\cdots,x_n\in \omega}{\neq}}
  f(x_1,\cdots,x_n,\omega)\bigg]
  =
  \zeta^n\int_{(\R^d)^n}
  \E^{\ssup{x_1,\cdots,x_n}}[f(x_1,\cdots,x_n,\omega)]
  \,\d x_1\cdots \d x_n
\end{equation}
for $f\in L^1((\R^d)^n\times \Omega)$, where the sign $\neq$ in the sum
indicates that the sum is taken over pairwise distinct elements. Also,
$\E^{\ssup{x_1,\cdots,x_n}}$ stands for expectation with respect to the
$n$-fold Palm distribution $\P^{\ssup{x_1,\cdots,x_n}}$.

\subsubsection{Continuum percolation.}

As in Section \ref{subsec pointPalm}, each $\omega\in\Omega$ is a locally
finite point set in $\R^d$. For any $\omega\in \Omega$, we define the random
open set
\begin{equation}
  \label{eq:Def.Cinfty}
  \mathcal{C}(\omega):=\bigcup_{x\in \omega}B_{\frac 12}(x) \subset \R^d.
\end{equation}
Here $B_r(x)=\{y\in \R^d:|y-x|<r\}$ denotes an open ball centered at $x$ of
radius $r>0$. The set $\mathcal{C}(\omega)$ can be decomposed into a disjoint
union of connected components. If there is a unique unbounded, open and
connected component, then this component is denoted by
$\mathcal C_\infty(\omega)\subset \mathcal{C}(\omega) \subset \R^d$. The
boundary of $\mathcal C_\infty(\omega)$ will be denoted by
$\partial \mathcal{C}_\infty(\omega)$. Moreover, we define
\begin{equation}
  \label{eq:Omega.0}
  \Omega_0:=\big\{\omega\in \Omega\colon \mathcal C_\infty(\omega)
  \text{ exists and } 0\in \mathcal C_\infty(\omega) \big\} \subset \Omega.
\end{equation}
If $ \P(\Omega_0)>0$ (which we will assume in condition
\ref{assump:inf-comp} stated below), then we can define the conditional
probability measure $\P_0$ on $\Omega_0$ by
\[
\P_0(A):=\P(A\mid \Omega_0)
=
\frac{\P(A\cap \Omega_0)}{\P(\Omega_0)},
\qquad A\in \mathcal G.
\]
By openness and connectedness of $\mathcal C_\infty(\omega)$ whenever it exists,
every two points $x,y\in \mathcal C_\infty(\omega)$ can be connected by a curve
in $C^1([0,1];\R^d)$. The interior distance $\d_\omega$ is defined on
$\mathcal C_\infty(\omega)$ by
\begin{align}\label{def domega}
  \d_\omega(x,y)=\inf\bigg\{\:\int_0^1|\dot r(s)|\d s \,\, \colon
  &\; r\in C^1([0,1];\R^d), \
  r(0)=x,\ r(1)=y,  \\
&\;\text{and }  r(s) \in \mathcal C_\infty(\omega) \text{ for all }s \in [0,1]
\bigg\}.
\end{align}
To state condition \ref{assump:exp-dec-dist-indshift} below, we define
$n(\omega,e)\in \N$ for all $e\in \Z^d$ with $|e|_1=1$ and
$\omega\in\Omega_0$ as the first arrival of $\mathcal C_\infty$ along the
direction $e$, namely
\begin{align}
  \label{def-n}
  n(\omega,e)=\min\{k\in\N:k e\in{\mathcal C_\infty}(\omega)\}.
\end{align}

\subsubsection{Assumptions on percolation.}\label{subsec assump perc}

\begin{enumerate}[label=\textbf{(P\arabic*)}]
\itemsep0.2em

\item $\P$ is stationary and ergodic with respect to $(\tau_x)_{x\in
    \R^d}$. Moreover, $\P$ is also ergodic with respect to $\tau_e$
  for all $e\in \Z^d$ with $|e|_1=1$; namely, any $A\in \mathcal{G}$ with
  $\tau_e A = A$ satisfies $\P(A)\in \{0,1\}$.
  \label{assump:est-erg}

\item The intensity measure $\Theta$ defined in \eqref{eq:Theta-def} satisfies
  $\Theta(A)<\infty$ for any compact set $A\subset \R^d$. In particular,
  $\Theta=\zeta \lambda$ for some $\zeta\in (0,\infty)$.
  \label{assump:intensity}

\item Recall the definition of $\mathcal C(\omega)$ from \eqref{eq:Def.Cinfty}
and that of $\Omega_0$ from \eqref{eq:Omega.0}. We assume that
$\P(\Omega_0)>0$. In other words, with positive $\P$-probability, the set
$\mathcal C(\cdot)\subset \R^d$ has a unique open, unbounded and connected
component $\mathcal C_\infty(\cdot)$ containing the origin $0\in \R^d$.
\label{assump:inf-comp}

\item
\begin{enumerate}
\item The Palm distribution $\P^{\ssup{x,y}}$ defined in
\eqref{eq:multidim-campbell-eq} and the distance $\d_\omega(x,y)$ defined in
\eqref{def domega} satisfy, for some $c_0,c_1,c_2>0$,
\begin{equation}
\label{eq:chem_dist_ineq}
\P^{\ssup{x,y}} \big(\d_\omega(x,y)\geq
  c_0|x-y|_{\infty}; 0,x,y\in {\mathcal C_\infty}\big)
  \leq c_1 \e^{-c_2 |x-y|_{\infty}}
  \qquad \forall x,y\in \R^d.
\end{equation}

\item We assume that there are constants $c_3,c_4>0$ such that for all
$\varrho>0$,
\begin{equation}\label{eq:exp-dec-dist-indshift}
\P_0\big( \,|\mathfrak{v}_e(\omega)|>\varrho \,\big)
\leq c_3\e^{- c_4 \varrho},
\qquad
\mathfrak v_e:=n(\omega,e) e,
\qquad
\forall\, e\in \Z^d \text{ with }|e|_1=1.
\end{equation}
\label{assump:exp-dec-dist-indshift}
\end{enumerate}
\label{assump:chem-dist}

\item The FKG inequality is satisfied. Concretely, suppose $A_1,A_2\subset
\Omega$ are increasing events, meaning that for any locally finite point sets
$\omega, \omega^\prime \in \Omega$ in $\R^d$, if
$\omega \subset \omega^\prime$ and $\omega \in A_i$, then
$\omega^\prime \in A_i$ for $i=1,2$. For any such events $A_1,A_2$, we assume
that
\[
{\P}(A_1\cap A_2)\geq {\P}(A_1){\P}(A_2).
\]
\label{assump:fkg}

\item Let $\d(0,\partial\mathcal C_\infty)$ denote the distance from $0$ to the
boundary of the infinite cluster $\mathcal C_\infty(\omega)$. Then we assume
that
\[
\E_0[\mathrm d(0,\partial\mathcal C_\infty)^{-\chi}]<\infty
\qquad\mbox{if and only if $\chi\in(0,1)$}.
\]
\label{assump:dist}

\end{enumerate}

We briefly explain the above assumptions. Conditions
\ref{assump:est-erg}-\ref{assump:inf-comp} are self-explanatory.
Assumption \ref{assump:chem-dist}(a) guarantees that $\d_\omega$ is comparable
to the Euclidean distance with high probability. Regarding
\ref{assump:exp-dec-dist-indshift}, we note that since $\P$ is ergodic with
respect to $\tau_e$ by \ref{assump:est-erg} and since
${\P}(0\in{\mathcal C_\infty})>0$, by the Poincar\'e recurrence theorem
(cf.\ \cite[Sec.\ 2.3]{P89}) we have $n(\omega,e)<\infty$.
Then \ref{assump:exp-dec-dist-indshift} implies that moving along the
coordinate axes has good recurrence properties and that $\mathfrak v_e$
possesses all moments under $\P_0$. The condition \ref{assump:fkg} is a useful
monotonicity inequality for percolation.

Assumption \ref{assump:dist} is of conceptual relevance, as it encodes a
precise relationship between the geometry of the continuum percolation cluster
and the structural properties of the diffusion matrix $a$, namely its
degeneracy, regularity, and integrability with respect to the conditioned
measure $\P_0$; see \ref{f1} below. In Theorem \ref{theorem 3} we will verify these assumptions.

\section{Main results}\label{sec-results}
This section is organized as follows: In Section  \ref{sec-main-results} we
introduce the structural assumptions required for the construction of the
quenched degenerate diffusion and state the main non-attainment theorem
precisely. In Section \ref{sec-HJB} we formulate the associated stochastic control
problem, introduce the optimal-control representation formula, and state the
corresponding viscosity-solution result. Finally, in Section \ref{sec perc models} we introduce
the continuum percolation models considered in this work and state the
verification theorem showing that the structural assumptions are satisfied for
these models.

\subsection{Construction of the quenched degenerate diffusion on $\mathcal C_\infty$.}\label{sec-main-results}

We now state the assumptions on the coefficients of the quenched degenerate diffusion to be constructed in Theorem \ref{theorem 1} below. Denote by $\mathcal{S}_d$ the space of $d\times d$ symmetric matrices. There is a natural partial order on $\mathcal{S}_d$: we say that for $A,B\in \mathcal{S}_d$, $A\leq B$ if $B-A$ is positive semidefinite, i.e., all its eigenvalues are nonnegative. For any symmetric positive semidefinite matrix $a$ (which will be defined below in \ref{f1}), denote by $\sigma\in \mathcal{S}_d$ the unique symmetric positive semidefinite matrix such that
\[
a=\frac12\sigma\sigma.
\]
We also define the inner product $\langle \cdot,\cdot\rangle_a=\langle\cdot,\cdot\rangle_{a(\omega)}$ and the corresponding seminorm by
\begin{equation}\label{eq:a-inner-prod}
\begin{aligned}
\langle v,w\rangle_a
&:=\langle a(\omega)v,w\rangle
=\langle v,a(\omega)w\rangle,
\\
\|v\|_a
&:=\sqrt{\langle v,v\rangle_a},
\qquad \forall\, v,w\in \R^d.
\end{aligned}
\end{equation}

We recall the notation of the underlying percolation models from Section \ref{sec-assume-models}.

\begin{enumerate}[label=\textbf{(A\arabic*)}]
\itemsep0.2em

\item \label{f1} ({\bf The diffusion coefficient})
\begin{enumerate}

\item
$a:\Omega\to\mathcal S_d$ is positive semidefinite and
$a(x,\omega):=a(\tau_x\omega)$ defines a stationary process with respect to the action of $\{\tau_x\}_{x\in\R^d}$ on $(\Omega,\mathcal G,\P)$; recall \eqref{def-stationary}. Moreover, for any $\omega\in\Omega_0=\{0\in\mathcal C_\infty\}\subset\Omega$,
\[
\supp(a(\cdot,\omega))
\subset
\overline{\mathcal C_\infty(\omega)}.
\]

\item
The maps
\[
x\mapsto a(x,\omega)=a(\tau_x\omega)
\qquad\mbox{and}\qquad
x\mapsto \xi(\tau_x\omega)
\]
are globally Lipschitz continuous, and the square root map
\[
x\mapsto \sigma(x,\omega)
\]
is locally Lipschitz away from the boundary.

\item
Moreover,
\[
\mathcal C_\infty(\omega)\ni x
\mapsto
\mathrm{div}\, a(x,\omega)\in\R^d
\]
is locally Lipschitz continuous away from the boundary, and
$|\mathrm{div}\,a|$ is uniformly bounded. 

\item
The restriction of $a$ to $\Omega_0$ satisfies the following: there exists
$c_5\in(0,\infty)$ and a measurable function
$\xi:\Omega_0\to(0,\infty)$ such that $\P_0$-almost surely,
\begin{equation}\label{eq:ellip-bounds}
\xi(\omega)|v|^2
\leq
\langle a(\omega)v,v\rangle
\leq
c_5|v|^2,
\qquad
\forall\, v\in\R^d.
\end{equation}

\end{enumerate}

\end{enumerate}

We refer to Section \ref{sec exmaple a} for a concrete example of
$a=a(x,\omega)$ satisfying the above properties. We also point out that the Lipschitz-order degeneracy imposed on $x\mapsto a(x,\omega)$ in Assumption~\ref{f1}(b) is natural from the point of view of boundary non-attainment. We refer to Remark~\ref{remark:sharpness} below for a model example showing that boundary attainment in Theorem \ref{theorem 1} may indeed occur below this Lipschitz threshold.

\begin{remark}\label{remark boundary}
In our follow-up paper \cite{BMM26}, in addition to \ref{f1} we assume that there exist $\delta>0$ and
\begin{equation}\label{alpha delta relation}
\alpha>2\left(\frac{1+\delta}{1-\delta}\right)>1+\delta
\end{equation}
(see also \eqref{eq:H1-H} below) such that the function $\xi(\cdot)$ in \eqref{eq:ellip-bounds} satisfies
\begin{equation}\label{eq:xi-mom-bound}
\E_0\big[\xi^{-\chi}\big]<\infty,
\qquad
\chi=\chi(\alpha,\delta)
:=
\frac{\alpha}{2}
\frac{1+\delta}{\alpha-(1+\delta)}.
\end{equation}
This moment condition on the degeneracy is closely related to the continuity assumptions \ref{f1}(b) and Assumption \ref{assump:dist}. Because of \eqref{alpha delta relation}
it follows that $\chi<1$. The lower bound in \eqref{eq:ellip-bounds} together with the Lipschitz condition \ref{f1}(b) imply that
$\xi(\omega)
\leq
\mathrm d(0,\partial\mathcal C_\infty(\omega))$.
Consequently, the moment condition \eqref{eq:xi-mom-bound} enforces
\[
\E_0\big[
\mathrm d(0,\partial\mathcal C_\infty)^{-\chi}
\big]
<\infty,
\]
which is consistent with Assumption \ref{assump:dist}; see Corollary \ref{cor thm3}. 
\qed
\end{remark}

To construct the quenched diffusion process, we fix a filtered probability space
$(\mathscr X,\mathcal F,\{\mathcal F_t\}_{t\ge0},P)$ together with a $d$-dimensional Brownian motion $(B_t)_{t\ge0}$ adapted to the filtration $(\mathcal F_t)_{t\ge0}$. We assume that the law $P$ of $(B_t)_{t\ge0}$ is independent of the law $\P$ of the point process discussed in Section \ref{sec-assume-models}. Let
\begin{equation}\label{def-class-C}
\begin{aligned}
\mathbf C_T
=
\bigg\{
c:[0,T]\times\mathscr X\to\R^d
\colon\;
&c \mbox{ is progressively measurable and}
\\
&E^P\bigg(
\int_0^T|c(s)|^2\,\d s
\bigg)
<\infty
\bigg\}.
\end{aligned}
\end{equation}
be the class of admissible controls.\footnote{
The integrability condition defining $\mathbf C_T$ is stronger than the
weighted condition naturally associated with the degenerate diffusion matrix
$a$. We retain this formulation for simplicity and compatibility with the
stochastic-control framework used below. See Remark~\ref{remark:weighted}
after the proof of Theorem~\ref{theorem 1} for further discussion.
} 
For any $a$ satisfying \ref{f1}, $\omega\in\Omega_0$,
$x\in\mathcal C_\infty(\omega)$ and $c\in\mathbf C_T$, we consider the diffusion process
\begin{equation}\label{eq:controlled-sde}
X_t
=
x
+
\int_0^t \sigma(X_s)\,\d B_s
+
\int_0^t (\mathrm{div}\,a)(X_s)\,\d s
+
\int_0^t a(X_s)c(s)\,\d s,
\qquad
\forall\, t\ge0,
\quad\mbox{a.s.}
\end{equation}
written equivalently as the stochastic differential equation
\begin{equation}\label{SDE}
\d X_t
=
\mathbf b(X_t,c_t)\,\d t
+
\sigma(X_t)\,\d B_t,
\qquad
\mbox{where}
\qquad
\mathbf b(y,c)
=
a(y)c+\mathrm{div}\,a(y).
\end{equation}

Although the above displays hold for every fixed $\omega\in\Omega_0$, we suppress the dependence on $\omega$ in the notation. We recall from \ref{f1} that for every $x\in\R^d$ and $\omega\in\Omega_0$,
\[
\sigma(x,\omega)=\sigma(\tau_x\omega),
\qquad
a(x,\omega)=a(\tau_x\omega),
\qquad
(\mathrm{div}\,a)(x,\omega)
=
(\mathrm{div}\,a)(\tau_x\omega).
\]

Here is our first main result.

\begin{theorem}\label{theorem 1}
Assume \ref{f1} and \ref{assump:est-erg}-\ref{assump:inf-comp}. For any $\omega\in\Omega_0$ and $c\in\mathbf C_T$, let
$P^{c,\omega}_x$ denote the law of the diffusion $(X_t)_{t\ge0}$ defined by \eqref{eq:controlled-sde}, and define
\[
\tau^\omega
:=
\inf\{t>0:X_t\in\partial\mathcal C_\infty(\omega)\}.
\]
Then for any $\omega\in\Omega_0$, $x\in\mathcal C_\infty(\omega)$, $T>0$, and $c\in\mathbf C_T$,
\[
P_x^{c,\omega}\big[\tau^\omega=\infty\big]=1.
\]
Moreover, for any $\omega\in\Omega_0$, $x\in\mathcal C_\infty(\omega)$, $T>0$, and $c\in\mathbf C_T$, there exists a unique strong solution of the stochastic differential equation \eqref{SDE}.
\end{theorem}

The proof of Theorem \ref{theorem 1} is given in Section \ref{sec proof thm1}. There, the non-attainment for the reversible diffusion, corresponding to the control
$c\equiv0$, is shown by developing a finite-energy barrier construction. The result for general $c\in\mathbf C_T$ is then obtained by a relative entropy estimate. Thus the confinement mechanism is already present at the level of the reversible degenerate diffusion, while admissible controls preserve this property by absolute continuity. We refer to Remark \ref{remark:sharpness} that underlines that the Lipschitz regularity of $x\mapsto a(x,\omega)$ in Assumption \ref{f1}(b) is important from the point of view of boundary non-attainment. Section \ref{sec exmaple a} provides an explicit construction of a diffusion matrix $a$ that satisfies the smoothness and positivity conditions assumed in \ref{f1}.

\subsection{Hamilton--Jacobi--Bellman equation and representation formula}\label{sec-HJB}

We next turn to the representation formula for viscosity solutions of the Hamilton--Jacobi--Bellman equation \eqref{eq HJB eps 1}, which will be established in Theorem \ref{theorem 2}. We impose the following assumptions on the coefficients of this equation.

\begin{enumerate}[label=\textbf{(H\arabic*)}]
\itemsep0.2em

\item \label{f2} ({\bf Hamiltonian: convexity and coercivity})
The Hamiltonian $H:\R^d\times\Omega\to\R$ satisfies, for each
$\omega\in\Omega$, that $p\mapsto H(p,\omega)$ is convex. Moreover, there are constants $c_6,\dots,c_9>0$ such that for all
$(p,\omega)\in\R^d\times\Omega_0$,
\begin{equation}\label{eq:H1-H}
c_6\|p\|_a^\alpha-c_7
\leq
H(p,\omega)
\leq
c_8\|p\|_a^\alpha+c_9,
\end{equation}
and $H(\cdot,\omega)\equiv 0$ outside $\Omega_0$. Here, $\alpha>1+\delta$ and $\delta>0$ are as in \eqref{eq:xi-mom-bound}. Equivalently, for
$\alpha^\prime:=\frac{\alpha}{\alpha-1}$ and constants $c_{10},\dots,c_{13}$,
\begin{equation}\label{eq:H1-L}
c_{10}\|q\|_a^{\alpha^\prime}-c_{11}
\leq
L(q,\omega)
\leq
c_{12}\|q\|_a^{\alpha^\prime}+c_{13},
\end{equation}
where
\begin{equation}\label{def L}
L(q,\omega):=\sup_{p\in\R^d}
\big[
\langle p,q\rangle_a-H(p,\omega)
\big].
\end{equation}
Moreover, $\R^d\ni x\mapsto L(q,\tau_x\omega)$ is continuous and, in particular,
$L(\cdot,\omega)\equiv 0$ outside $\Omega_0$.

\item \label{f3} ({\bf Hamiltonian: stationarity and continuity})
For any $x\in\R^d$, we define
\[
H(x,p,\omega):=H(p,\tau_x\omega),
\qquad
L(x,q,\omega):=L(q,\tau_x\omega).
\]
Recall that by \ref{assump:est-erg}, $\P$ is stationary with respect to
$\{\tau_x\}_{x\in\R^d}$. Hence the maps
\[
x\mapsto H(x,p,\omega)=H(p,\tau_x\omega),
\qquad
x\mapsto L(x,q,\omega)=L(q,\tau_x\omega)
\]
define stationary processes with respect to $\{\tau_x\}_{x\in\R^d}$ and $\P$.

Next, we assume that there are constants $c_{14},c_{15}>0$ such that, for any
$\omega\in\Omega_0$, $x,y\in\R^d$ and $p\in\R^d$,
\begin{equation}\label{eq:H2}
|H(x,p,\omega)-H(y,p,\omega)|
\leq
(c_{14}|p|^\alpha+c_{15})|x-y|.
\end{equation}

\item \label{f4} ({\bf Initial condition})
The initial condition $f:\R^d\to\R$ is uniformly continuous. In particular, for any $\delta>0$, there exists $K_\delta>0$ such that for any $x,y\in\R^d$,
\begin{equation}
\label{eq:H6}
|f(x)-f(y)|
\leq
K_\delta |x-y|+\delta.
\end{equation}

\end{enumerate}

\noindent\begin{example}\label{exa HL} We provide a class of admissible Hamiltonians satisfying 
\ref{f2}-\ref{f3} which depends on the diffusion matrix only through the quadratic form induced 
by $a$. The construction of a diffusion matrix $a$ satisfying \ref{f1} is given in 
Corollary \ref{cor example a}. The usage of $a$  is important because $x\mapsto a(x,\omega)$ 
is globally Lipschitz, whereas the square root $x\mapsto \sigma(x,\omega)$ is only locally 
Lipschitz away from the boundary. Thus one should not define $H$ directly in terms of 
$\sigma(x,\omega)$.

For $\alpha\geq 2$, a basic example is
\[
H(x,p,\omega)
=
\bigl(1+\langle p,a(x,\omega)p\rangle\bigr)^{\alpha/2}
\asymp 1+\|p\|_{a(x,\omega)}^\alpha,
\]
and therefore the coercivity and growth assumptions in \ref{f2} are satisfied, up to additive constants.
Moreover, since $x\mapsto a(x,\omega)$ is globally Lipschitz,
\[
\big|
\langle p,a(x,\omega)p\rangle
-
\langle p,a(y,\omega)p\rangle
\big|
\leq C|p|^2|x-y|.
\]
For $\alpha\geq2$, this implies
\[
|H(x,p,\omega)-H(y,p,\omega)|
\leq
C(1+|p|^\alpha)|x-y|,
\]
which is the spatial continuity estimate required in \ref{f3}.

The associated Lagrangian is defined by the Legendre transform with respect to the weighted pairing
\[
\langle p,q\rangle_{a(x,\omega)}
=
\langle p,a(x,\omega)q\rangle.
\]
Since $H$ has superlinear growth in the seminorm $\|\cdot\|_{a(x,\omega)}$, the corresponding Lagrangian has the dual growth
\[
L(x,q,\omega)\asymp \|q\|_{a(x,\omega)}^{\alpha'}
\]
up to additive constants, where $\alpha'=\alpha/(\alpha-1)$.
\end{example}

Let us now fix the diffusion coefficient $a$ satisfying \ref{f1}, the Hamiltonian $H$ and teh Lagrangian $L$ satisfying \ref{f2} and \ref{f3}, and the initial condition $f$ satisfying \ref{f4}. For any fixed
$\omega\in\Omega_0$, $x\in\mathcal C_\infty(\omega)$, and
$c\in\mathbf C_T$, we define
\begin{align}
J_f^\omega(t,x,c)
&:=
E^{P_x^{c,\omega}}
\bigg[
f(X_t)
-
\int_0^t L(X_s,c(s),\omega)\,\d s
\bigg],
\label{eq:cost-fun-def}
\\
u(t,x,\omega)
&:=
\sup_{c\in\mathbf C_T}
J_f^\omega(t,x,c).
\label{eq:visc-sol-var-form}
\end{align}
Here, as in Theorem \ref{theorem 1}, $P_x^{c,\omega}$ denotes the law of the diffusion $(X_t)_{t\ge0}$ starting at $x\in\mathcal C_\infty(\omega)$.

For any $\eps>0$, define
\[
f_\eps(x):=\frac{1}{\eps}f(\eps x).
\]
We define the rescaled version of $u(t,x,\omega)$ by
\begin{equation}\label{eq:u-eps-def}
u_\eps(t,x,\omega)
:=
\sup_{c\in\mathbf C_T}
\eps\,
J^\omega_{f_\eps}
\bigg(
\frac t\eps,
\frac x\eps,
c
\bigg).
\end{equation}

Here is our next main result.

\begin{theorem}\label{theorem 2}
Assume \ref{f1}, \ref{f2}-\ref{f4} and
\ref{assump:est-erg}-\ref{assump:inf-comp}. Then the following hold:
\begin{enumerate}

\item
For any $\omega\in\Omega_0$ and $T>0$, the function
$u(t,x,\omega)$ defined in \eqref{eq:visc-sol-var-form} is a viscosity solution
(cf. Definition \ref{def:SubSuperSol}) of
\begin{equation}\label{eq HJB eps 1}
\begin{cases}
\partial_t u
=
\frac12\mathrm{div}\big(a(x,\omega)\nabla u\big)
+
H(x,\nabla u,\omega),
&\text{in } (0,T)\times\mathcal C_\infty(\omega),
\\
u(0,x,\omega)=f(x),
&\text{on } \mathcal C_\infty(\omega),
\end{cases}
\end{equation}
in the class of viscosity solutions of at most linear growth.

\item
For any $\eps>0$, the function $u_\eps$ defined in \eqref{eq:u-eps-def} is a viscosity solution of
\begin{equation}\label{eq-HJB'}
\begin{cases}
\partial_t u_\eps
=
\frac{\eps}{2}
\mathrm{div}\big(a\big(\frac x\eps,\omega\big)\nabla u_\eps\big)
+
H\big(\frac x\eps,\nabla u_\eps,\omega\big),
&\text{in } (0,T)\times\eps\mathcal C_\infty(\omega),
\\
u_\eps(0,x,\omega)=f(x),
&\text{on } \eps\mathcal C_\infty(\omega),
\end{cases}
\end{equation}
in the class of viscosity solutions of at most linear growth.

\item
If we additionally assume that the initial condition is uniformly bounded, i.e.,
$\sup_{x\in\R^d}|f(x)|\leq K_f<\infty$, then for each $T>0$ and
$\omega\in\Omega_0$,
\begin{equation}
\sup_{t\in[0,T]}
\sup_{\eps>0}
\sup_{x\in\eps\mathcal C_\infty(\omega)}
|u_\eps(t,x,\omega)|
<\infty.
\end{equation}

\end{enumerate}
\end{theorem}

The proof of Theorem \ref{theorem 2} is given in Section \ref{sec proof thm2}.

\subsection{Verification result for continuum percolation}\label{sec perc models}

We now introduce two important continuum percolation models and state the verification result for the assumptions introduced above.

\subsubsection{The Boolean model.}\label{sec Boolean}

The simplest continuum percolation model, known as the Boolean model, is defined by
\[
\mathcal C(\omega):=\bigcup_{x\in\omega}B_{1/2}(x),
\]
where $\omega$ is sampled according to a probability measure $\P=\P^\zeta$ satisfying the following properties; recall the notation from Section \ref{sec-assume-models}:
\begin{itemize}
\item For any bounded $A\subset\R^d$ and any $n\in\N_0$,
\begin{equation}\label{Poisson}
\P^\zeta(\#(\omega\cap A)=n)
=
\frac{(\zeta |A|)^n}{n!}
\e^{-\zeta|A|}.
\end{equation}

\item For any collection of disjoint bounded Borel sets
$A_1,\ldots,A_k\subset\R^d$,
\begin{equation}\label{iidPPP}
\P^\zeta
\big(
\#(\omega\cap A_1)=n_1,\ldots,\#(\omega\cap A_k)=n_k
\big)
=
\prod_{i=1}^k
\P^\zeta(\#(\omega\cap A_i)=n_i).
\end{equation}
\end{itemize}

\subsubsection{Continuum random cluster model.}\label{sec cluster model}

Another important continuum percolation model is the continuum random cluster model (CRCM) \cite{DH15,H18}. Unlike the Boolean model, this model exhibits long-range correlations. For $\omega\in\Omega$ and $\Lambda\subset\R^d$, set
\[
\omega_\Lambda:=\omega\cap\Lambda.
\]
A probability measure $P$ on $(\Omega,\mathcal G)$ is a continuum random cluster model with parameters $q\geq1$ and $\zeta>0$, denoted by CRCM$(q,\zeta)$, if it is stationary and, for $P$-almost every configuration $\omega$ and each bounded set $\Lambda\subset\R^d$, the conditional law of $P$ given $\omega_{\Lambda^c}$ is absolutely continuous with respect to the Poisson point process restricted to $\Lambda$, denoted by $\P^\zeta_\Lambda$, with density
\[
\frac{
q^{N^\Lambda_{cc}(\cdot\cup\omega_{\Lambda^c})}
}{
Z_\Lambda(\omega_{\Lambda^c})
}.
\]
Here, $N^\Lambda_{cc}$ is the $\Lambda$-local number of connected components of a configuration \cite[Definition 2.1]{H18}, and $Z_\Lambda$ is the partition function
\[
Z_\Lambda(\omega_{\Lambda^c})
:=
\int_\Omega
q^{N^\Lambda_{cc}(\omega'_\Lambda\cup\omega_{\Lambda^c})}
\,\P^\zeta_\Lambda(\d\omega'_\Lambda).
\]
Equivalently, the DLR equations are satisfied: for every bounded measurable function $f$ and bounded set $\Lambda\subset\R^d$,
\[
\int_\Omega f(\omega)\,\d P
=
\int_\Omega\int_\Omega
f(\omega'_\Lambda\cup\omega_{\Lambda^c})
\frac{
q^{N^\Lambda_{cc}(\omega'_\Lambda\cup\omega_{\Lambda^c})}
}{
Z_\Lambda(\omega_{\Lambda^c})
}
\,\P^\zeta_\Lambda(\d\omega'_\Lambda)
\,P(\d\omega).
\]

\begin{theorem}\label{theorem 3}
Fix $d\geq2$. Then there exists $\zeta_c\in(0,\infty)$ such that for all
$\zeta>\zeta_c$, the Boolean model $(\Omega,\mathcal G,\P^\zeta)$ satisfies
\ref{assump:est-erg}-\ref{assump:dist} and, if $d\geq3$, also satisfies
\ref{assump:exp-dec-dist-indshift}. Moreover, the continuum random cluster model satisfies the assumptions \ref{assump:est-erg}-\ref{assump:inf-comp}.
\end{theorem}

The proof of Theorem \ref{theorem 3} is given in Section \ref{sec proof thm3}. While we do not pursue it here, we expect that the continuum random cluster model can also be shown to satisfy the remaining assumptions \ref{assump:chem-dist}-\ref{assump:dist}; see Section \ref{subsec sketch P4} for an outline of the verification of \ref{assump:chem-dist} for this model.

In light of Remark \ref{remark boundary}, we record the following immediate consequence of Theorem \ref{theorem 3}: 

\begin{cor}\label{cor thm3}
Assume \ref{assump:dist}. Then 
\[ 
\E_0[d(0,\partial\mathcal C_\infty)^{-\chi}]<\infty \qquad \text{for every }  \chi \in (0,1).  
\] 
Consequently, if $\xi$ satisfies 
\[ 
\xi(\omega)  {}\gtrsim{}  d(0,\partial\mathcal C_\infty(\omega)), 
\] 
then 
\[ \E_0[\xi^{-\chi}]<\infty \qquad \text{for every }  \chi \in (0,1).  
\] 
This integrability regime is precisely the one required in the homogenization theory developed in the companion paper \cite[Assumption (A2)]{BMM26}. In particular, Theorem \ref{theorem 3} implies that this condition is satisfied for the supercritical Boolean continuum percolation model.
\end{cor}

\subsection{Related literature and ideas of proof}
The study of diffusions and homogenization in random environments has a long history, particularly in uniformly elliptic settings or on the full space $\R^d$; see, for instance,
\cite{PV81,PV82,K85,KV86,LPV87,JKO94}. More recently, stochastic homogenization for Hamilton--Jacobi equations and viscous Hamilton--Jacobi equations in stationary ergodic environments has been developed in a variety of settings, including degenerate regimes; see among many others
\cite{I99,RT00,So99,LS05,LS10,KRV06,KV08,K08}. Random walks on percolation structures have also been studied extensively, both in the discrete setting 
\cite{B04,BB07,MP07,SS04,B11,PRS15,D19} (see also \cite{ADS15,A25} for studies on degenerate environments) and in continuum random media; see for instance
\cite{S94,OS95,CD16,T24}. Geometric and probabilistic properties of continuum percolation itself are by now classical; see in particular \cite{MR96}.

The present work differs from these frameworks because the underlying state space is a genuine continuum percolation cluster and the diffusion coefficient is allowed to degenerate near the random boundary. In particular, the geometry of the domain is itself random, highly irregular, and conditioned on the existence of an infinite connected component. This leads to several difficulties due to the interaction between the degeneracy of the diffusion and the geometry of the cluster near the boundary.

A central difficulty throughout the paper is that the random boundary
$\partial\mathcal C_\infty(\omega)$
lacks any smooth geometric structure. In particular, classical approaches based on smooth boundary analysis, reflection arguments, or It{\^o}-type localization near the boundary are not naturally available in the present setting. While one could in principle study reflected or obliquely reflected diffusions on such random domains, this would require additional control of boundary local times and boundary interactions in a geometry with very limited regularity. The approach developed here instead avoids introducing boundary dynamics altogether.

That said, this does not remove the influence of the boundary geometry from the problem. Indeed, the regularity assumptions imposed on the diffusion coefficient must remain compatible with quantitative integrability properties of the degeneracy, which are required for large-scale analysis in random media. In continuum percolation these conditions are constrained by the behavior of the distance-to-the-boundary function under the conditioned measure $\P_0$, whose inverse moments are themselves limited by the geometry of the infinite cluster. Thus, although the diffusion avoids direct boundary interaction, the irregular boundary geometry still enters fundamentally through the admissible degeneracy regime.
One of the main contributions of the paper is to identify a regime in which these competing requirements remain simultaneously compatible.

The proof of Theorem \ref{theorem 1} exploits the degeneracy of the diffusion coefficient near the boundary through a finite-energy logarithmic barrier argument. Rather than controlling the diffusion through geometric regularity of the boundary, the idea is that the degeneracy itself prevents the process from reaching $\partial\mathcal C_\infty(\omega)$. Concretely, a noteworthy feature here is the appearance of logarithmic barriers
of the form 
$$
\bigg(\log \frac 1 {d(x,\partial\mathcal C_\infty(\omega))}\bigg)^\varepsilon \qquad \mbox{with}\qquad \varepsilon<1/2.
$$
The threshold
$\varepsilon=1/2$ marks a critical regime in which the boundary degeneracy and
the singular behavior of the barrier exactly balance. This borderline
phenomenon plays a central role in controlling the diffusion near the random
boundary of the continuum percolation cluster. The arguments are developed in Section \ref{subsec barrier}- \ref{subsec entropy}. Section \ref{sec exmaple a} provides an explicit construction of the diffusion matrix $a$ satisfying \ref{f1} and Remark \ref{remark:sharpness} shows that the Lipschitz continuity in \ref{f1}(b) is quite natural to for non-attainment of boundary. 

Theorem \ref{theorem 1} yields a global quenched diffusion process without imposing regularity assumptions on the random boundary beyond the geometric assumptions encoded in the framework. This is also the key input in the proof of Theorem \ref{theorem 2}: since trajectories remain in the interior of the cluster for all times almost surely, the stochastic control problem and the associated dynamic programming principle can be formulated without imposing boundary conditions on the random domain. This yields a quenched stochastic representation formula for viscosity solutions of degenerate Hamilton--Jacobi--Bellman equations on continuum percolation clusters. The proof of Theorem \ref{theorem 2} is given in Section \ref{sec proof thm2} by leveraging the assumptions \ref{f2}-\ref{f4} combined with \ref{f1}.

Finally, Theorem~\ref{theorem 3} shows that the assumptions introduced in the
paper are simultaneously realizable in concrete continuum percolation models.
In particular, for the Boolean model we obtain quantitative control on the tail distribution of the
distance-to-the-boundary function under the conditioned measure $\P_0$. Concretely, combining local Poisson geometry with global connectivity properties of the
supercritical phase, it is shown that 
for all $\eps>0$ sufficiently small,
\[
c\,\eps
\le
\P_0\big(\dist\bigl(0,\partial\mathcal C_\infty(\omega)\bigr) <\eps\big )
\le
C\,\eps,
\]
see Section \ref{subsec proof P6}. This implies the inverse distance-to-the-boundary function being
integrable under $\P_0$ up to exponents strictly smaller than one, and in
general not beyond this threshold. This result is of particular conceptual
importance for the homogenization theory developed in the companion
paper~\cite{BMM26}: Indeed, as noted in Remark \ref{remark boundary}, the geometry of the continuum percolation cluster imposes the intrinsic restriction $\chi<1$, while the admissible moment
condition on the degeneracy is determined jointly by this geometric constraint
and the coercivity assumptions \ref{f2} on the Hamiltonian; we refer to
\cite[Remark~1]{BMM26} for details on this. In this sense, the geometry of the
random cluster directly influences the range of admissible degeneracies and
thus the structure of the effective theory.

Verifying the structural assumptions for the continuum random cluster model
further indicates that the framework is not restricted to finite-range or
i.i.d.\ environments, but naturally extends to models with long-range
correlations. The details can be found in
Section~\ref{sec proof thm3}.

To the best of our knowledge, quenched stochastic control representations and Hamilton--Jacobi--Bellman equations for degenerate controlled diffusions on genuine continuum percolation clusters have not previously been studied in the literature. The results obtained here also provide the foundational probabilistic and analytic framework for the homogenization theory developed in the companion paper \cite{BMM26}.

\section{Proof of Theorem \ref{theorem 1}}\label{sec proof thm1}

The goal of this section is to prove Theorem \ref{theorem 1}, which states that for every fixed $\omega\in\Omega_0$, the diffusion $X_t$ stays inside the cluster $\mathcal C_\infty(\omega)$ for all times. We recall from Section \ref{sec-main-results} that $(\mathscr{X},\mathcal{F},\{\mathcal F_t\}_{t\geq 0},P)$ is a filtered probability space carrying an auxiliary $d$-dimensional Brownian motion $(B_t)_{t\geq 0}$ adapted to $(\mathcal F_t)_{t\geq 0}$. The law $P$ of this Brownian motion is independent of the law $\P$ of the point process introduced in Section \ref{sec-assume-models}. From \eqref{def-class-C}, we have the class
\[
\mathbf C_T
=
\bigg\{
c:[0,T]\times \mathscr X \to \R^d
\colon
c \mbox{ is progressively measurable and }
E^P\bigg(\int_0^T |c(s)|^2\,\d s\bigg)<\infty
\bigg\}.
\]

For any diffusion coefficient $a:\R^d\times\Omega_0\to\mathcal S_d$ satisfying Assumption \ref{f1} and any $c\in\mathbf C_T$, we consider the SDE
\begin{equation}\label{controlled-sde}
X_t
=
x
+
\int_0^t \sigma(X_s)\,\d B_s
+
\int_0^t(\mathrm{div}\,a)(X_s)\,\d s
+
\int_0^t a(X_s)c(s)\,\d s,
\qquad
\mbox{a.s. for all } t\geq0.
\end{equation}

\subsection{Barrier estimates}\label{subsec barrier}

The proof of Theorem \ref{theorem 1} is split into several steps contained in Lemmas \ref{lemma Dirichlet}-  \ref{cor Dirichlet 4}. The following lemma is the first key step.

\begin{lemma}\label{lemma Dirichlet}
Let $(B,\Sigma,\mu)$ be a probability space and let $L$ be the generator of a symmetric Markov process $(X_t)_{t\geq0}$ taking values in $B$, with invariant probability measure $\mu$. Let $P_x$ denote the law of $(X_t)_{t\geq0}$ starting from $x\in B$ and let
\[
P_\mu(\cdot):=\int_B P_x(\cdot)\,\mu(\d x).
\]
For a function $f$ in the domain of $L$, define the Dirichlet energy by
\[
\mathcal E(f,f):=-\int_B (Lf)(x)f(x)\,\mu(\d x).
\]
Then, for every $T>0$ and $\ell>0$,
\begin{equation}\label{est lemma Dirichlet}
P_\mu\bigg[\sup_{0\leq t\leq T}|f(X_t)|\geq \ell\bigg]
\leq
\frac{\e}{\ell}
\sqrt{\|f\|_{L^2(\mu)}^2+T\,\mathcal E(f,f)}.
\end{equation}

\end{lemma}

\begin{proof}
By replacing $f$ with $f/\ell$ and the generator $L$ with $TL$, it suffices to prove the claim for $\ell=T=1$, namely
\begin{equation}\label{lT1}
P_\mu\bigg[\sup_{0\leq t\leq 1}|f(X_t)|\geq 1\bigg]
\leq
\e\sqrt{\|f\|_{L^2(\mu)}^2+\mathcal E(f,f)}.
\end{equation}

Let
\[
A:=\{x\in B  : |f(x)|\geq 1\} \qquad\mbox{and}\qquad \tau:=\inf\{t>0:X_t\in A\}.
\]
For $\lambda>0$, the function
\[
U_\lambda(x):=E_x[\e^{-\lambda\tau}]
\]
is the equilibrium potential of $A$ for the resolvent parameter $\lambda$. In particular, for $\lambda=1$, the function
\[
U(x):=U_1(x)=E_x[\e^{-\tau}]
\]
satisfies $U=1$ on $A$, $0\leq U\leq1$, and minimizes the variational problem
\[
\inf_{g=1\ \mathrm{on}\ A}
\big[
\|g\|_{L^2(\mu)}^2+\mathcal E(g,g)
\big].
\]
Since $|f|\geq1$ on $A$, 
\begin{equation}\label{bound Uf}
\|U\|_{L^2(\mu)}^2+\mathcal E(U,U)
\leq
\|f\|_{L^2(\mu)}^2+\mathcal E(f,f).
\end{equation}
Moreover,
\[
P_x[\tau<1]
=
P_x[\e^{-\tau}>\e^{-1}]
\leq
\e\,E_x[\e^{-\tau}]
=
\e\,U(x).
\]
Integrating with respect to $\mu$ and using \eqref{bound Uf}, we obtain
\[
\begin{aligned}
P_\mu\bigg[\sup_{0\leq t\leq1}|f(X_t)|\geq1\bigg]
&=
P_\mu[\tau<1]
=
\int_E P_x(\tau<1)\,\mu(\d x)
\leq
\e\int_E U(x)\,\mu(\d x)
\\
&\leq
\e\,\|U\|_{L^2(\mu)}
\leq
\e\sqrt{\|U\|_{L^2(\mu)}^2+\mathcal E(U,U)}
\leq
\e\sqrt{\|f\|_{L^2(\mu)}^2+\mathcal E(f,f)}.
\end{aligned}
\]
This proves \eqref{lT1} and hence the lemma.
\end{proof}

\begin{lemma}\label{cor Dirichlet 1}
Let $G\subset\R^d$ be open and bounded, and let $(X_t)_{t\geq0}$ be a symmetric Markov process as in Lemma \ref{lemma Dirichlet}, with invariant measure $\mu$ equal to the normalized Lebesgue measure on $G$. Assume that the sample paths are continuous. Suppose there exists a function $f:G\to\R$ such that
\begin{enumerate}
\item $|f(x)|\to\infty$ as $x\to\partial G$,\label{1}
\item $\|f\|_{L^2(\mu)}^2+\mathcal E(f,f)<\infty$.\label{2}
\end{enumerate}
Then, for $\mu$-almost every starting point $x\in G$, we have 
\begin{equation}\label{claim est cor1}
P_x(\tau_{\partial G}<\infty)=0,
\qquad
\mbox{where}
\qquad
\tau_{\partial G}:=\inf\{t\geq0:X_t\in\partial G\}.
\end{equation}
\end{lemma}

\begin{proof}
Fix $T>0$. From Lemma \ref{lemma Dirichlet}, since
$\|f\|_{L^2(\mu)}^2+\mathcal E(f,f)<\infty$,
the right-hand side of the estimate \eqref{est lemma Dirichlet} tends to $0$ as $\ell\to\infty$. Therefore
\[
P_\mu\bigg[
\sup_{0\leq t\leq T}|f(X_t)|=+\infty
\bigg]=0.
\]
Because $|f(x)|\to\infty$ as $x\to\partial G$, for every $M>0$ there exists $\delta>0$ such that
\[
\mathrm{dist}(x,\partial G)<\delta
\quad\Longrightarrow\quad
|f(x)|>M.
\]
On the event $\{\tau_{\partial G}\leq T\}$, continuity of the sample paths implies that there exists a sequence $t_n\uparrow\tau_{\partial G}$ such that $X_{t_n}\to\partial G$. Hence $|f(X_{t_n})|\to\infty$, and therefore $\sup_{0\leq t\leq T}|f(X_t)|=+\infty$.
Thus
\[
\{\tau_{\partial G}\leq T\}
\subset
\bigg\{
\sup_{0\leq t\leq T}|f(X_t)|=+\infty
\bigg\}.
\]
Consequently,$P_\mu(\tau_{\partial G}\leq T)=0$.
Since $T>0$ is arbitrary,
\[
P_\mu(\tau_{\partial G}<\infty)
=
P_\mu\bigg(\bigcup_{n=1}^\infty\{\tau_{\partial G}\leq n\}\bigg)
\leq
\sum_{n=1}^\infty P_\mu(\tau_{\partial G}\leq n)
=0.
\]
Finally, since $P_\mu(\cdot)=\int_G P_x(\cdot)\,\mu(\d x)$,
we obtain
\[
0
=
P_\mu(\tau_{\partial G}<\infty)
=
\int_G P_x(\tau_{\partial G}<\infty)\,\mu(\d x).
\]
Hence $P_x(\tau_{\partial G}<\infty)=0$ for $\mu$-almost every $x\in G$.
\end{proof}

\begin{lemma}\label{cor Dirichlet 2}
Let $G\subset\R^d$ be open and bounded, and let
\[
L=\frac12\operatorname{div}(a\nabla)
\]
be the generator of a symmetric Markov process $(X_t)_{t\geq0}$ taking values in $G$, where $x\mapsto a(x)$ is a symmetric positive semidefinite matrix which is continuous and bounded, vanishes on $\partial G$, and is locally elliptic in the sense that
\[
\langle a(x)v,v\rangle\geq \xi(x)|v|^2,
\qquad
\inf_{x\in K}\xi(x)\geq\delta_K>0
\]
for every compact subset $K\subset G$. Suppose there exists a function $f:G\to\R$ satisfying \ref{1}--\ref{2} in Lemma \ref{cor Dirichlet 1}. Then, for every starting point $x\in G$,
\[
P_x(\tau_{\partial G}<\infty)=0.
\]
\end{lemma}

\begin{proof}
We write $\tau=\tau_{\partial G}$. By Lemma \ref{cor Dirichlet 1},
\[
P_x(\tau<\infty)=0
\qquad
\mbox{for $\mu$-almost every }x\in G.
\]
Define
\[
u(x):=P_x(\tau<\infty) \quad \text{for } x \in \overline G. 
\]
By the strong Markov property, $u$ is $L$-harmonic in $G$ in the probabilistic sense. Since $a$ is continuous and locally elliptic in $G$, standard elliptic regularity for locally elliptic operators implies that $u$ is continuous in $G$; see Stroock and Varadhan \cite{SV79}, or \cite[p.~65 and Ch.~VI]{Ba98}.

Since $u=0$ for $\mu$-almost every $x\in G$, and since $\mu$ is the normalized Lebesgue measure on $G$, every nonempty open subset of $G$ has positive $\mu$-measure. If there existed $x_0\in G$ such that $u(x_0)>0$, then by continuity there would exist $r>0$ and $\varepsilon>0$ such that
\[
u(x)\geq\varepsilon
\qquad
\mbox{for all }x\in B(x_0,r)\subset G.
\]
This contradicts $u=0$ $\mu$-almost everywhere. Hence $u(x)=0$ for every $x\in G$, which proves the lemma.
\end{proof}

\begin{lemma}\label{cor Dirichlet 3}
Let $G$, $\mu$, $L$ and $a$ be as in Lemma \ref{cor Dirichlet 2}. Assume in addition that $x\mapsto a(x)$ is globally Lipschitz continuous and that $G$ satisfies the boundary-layer estimate
\[
\big|\{x\in G:\mathrm d(x,\partial G)<r\}\big|
\leq C_G r
\qquad
\mbox{for all sufficiently small }r>0.
\]
 Set $\log^+(r):=\max\{0,\log r\}$, then  for every $\eps\in(0,\frac12)$ the function
\begin{equation}\label{example}
f(x) =
\bigg(  \log^+ \frac1{\mathrm d(x,\partial G)}\bigg)^\eps
\end{equation}
satisfies
\[
|f(x)|\to\infty
\qquad\mbox{as }x\to\partial G \qquad \text{and} \qquad 
\|f\|_{L^2(\mu)}^2+\mathcal E(f,f)<\infty.
\]
Consequently, $(X_t)_{t\geq0}$ starting from any $x\in G$ does not hit $\partial G$ in finite time almost surely.
\end{lemma}

\begin{proof}
Let
\[
d(x):=\mathrm d(x,\partial G).
\]
Since $a$ is globally Lipschitz and vanishes on $\partial G$, there exists $C>0$ such that
$|a(x)|\leq C d(x)$. Indeed, for any $x\in G$ and any $y\in\partial G$ with $|x-y|=d(x)$, we have $a(y)=0$ and hence $|a(x)|=|a(x)-a(y)|\leq C|x-y|=C d(x)$.

We now estimate the energy of $f$. Since $d$ is Lipschitz and $|\nabla d|=1$ almost everywhere, the function $f$ is weakly differentiable near the boundary and
\[
|\nabla f(x)|
\leq
C
\frac{1}{d(x)}
\bigg(\log^+\frac1{d(x)}\bigg)^{\eps-1}
\]
for $d(x)$ sufficiently small. Thus combining with $|a(x)|\leq C d(x)$  we have an estimate on the Dirichlet form
\begin{equation}\label{Dir}
\mathcal E(f,f)
=\frac12\int_G\langle a(x)\nabla f(x),\nabla f(x)\rangle\,\d x \leq
C
\int_G
\frac1{d(x)}
\bigg(\log^+\frac1{d(x)}\bigg)^{2\eps-2}
\,\d x.
\end{equation}
It remains to check that the latter integral is finite, for which it clearly suffices to estimate
\[
I := \int_{\{d(x)<1/2\}} \frac1{d(x)}
\bigg(\log\frac1{d(x)}\bigg)^{2\eps-2}
\,\d x.
\]
For $r\in(0,1)$, define
\[
A_r
:=
\{x\in G:d(x)<r\} \qquad\mbox{so that}\qquad |A_r|
\leq
C_G r
\]
by assumption. We decompose the boundary layer dyadically: For $k\geq0$, set
\[
E_k
:=
\big\{ x\in G : 
2^{-(k+1)} \leq d(x) < 2^{-k} \big\} 
\qquad\mbox{so that} \qquad |E_k|
\leq 
|A_{2^{-k} }|
\leq
C_G\,2^{-k}.
\]
Moreover, on $E_k$,
\[
\frac1{d(x)}
\bigg(\log\frac1{d(x)}\bigg)^{2\eps-2}
\leq
C 2^k 
\bigg( \log {2^k} \bigg)^{2\eps-2}.
\]
Therefore,
\[
I =
\sum_{k=1}^\infty
\int_{E_k} \frac1{d(x)}
\bigg(\log\frac1{d(x)}\bigg)^{2\eps-2}  \!\d x
\leq
C
\sum_{k=0}^\infty
2^k 
\Big( k\log 2\Big)^{2\eps-2}
|E_k|
\leq C \sum_{k=1}^\infty k^{2\eps-2}
\]
which converges precisely when
$2\eps-2<-1$, that is, for 
$\eps\in(0,\tfrac12)$. Thus the Dirichlet energy \eqref{Dir} is finite in this regime. 
Similarly,
\[
\|f\|_{L^2(\mu)}^2
\leq
C\int_G
\bigg(\log\frac1{d(x)}\bigg)^{2\eps}
\,\d x
<\infty,
\]
again by the previous argument, since logarithmic divergences are integrable against a boundary layer of order $r$. Thus $\|f\|_{L^2(\mu)}^2+\mathcal E(f,f)<\infty$ and 
the conclusion now follows from Lemma \ref{cor Dirichlet 2}.
\end{proof}

\begin{lemma}\label{cor Dirichlet 4}
The conclusion of Lemma \ref{cor Dirichlet 3} holds for the continuum percolation domains considered in this paper.
\end{lemma}

\begin{proof}
Fix $\omega\in\Omega_0$ and set $G=\mathcal C_\infty(\omega)$. For $R>0$, let
\[
G_R:=G\cap B_R.
\]
Since the point configuration is locally finite, $G_R$ is contained in a finite union of balls and therefore satisfies the boundary-layer estimate required in Lemma \ref{cor Dirichlet 3}. Applying Lemma \ref{cor Dirichlet 3} to $G_R$, we obtain that the diffusion does not hit $\partial G_R$ in finite time before leaving $G_R$.

Suppose now that the diffusion starting from $x\in G$ hits $\partial G$ at some finite time. By continuity of the path, the trajectory up to this hitting time is contained in $B_R$ for some sufficiently large $R$. Hence it would hit $\partial G_R$ in finite time, contradicting the conclusion for $G_R$. Letting $R\uparrow\infty$ proves the claim.
\end{proof}

\subsection{Concluding the proof of Theorem \ref{theorem 1}.}\label{subsec entropy}
By Lemmas \ref{cor Dirichlet 1}--\ref{cor Dirichlet 4}, we have
\[
\tau^\omega=\infty
\qquad
\mbox{almost surely under }P=P_x^{0,\omega},
\]
that is, for the diffusion with control $c\equiv0$. We now show that
\[
\tau^\omega=\infty
\qquad
\mbox{almost surely under }P^c=P_x^{c,\omega}
\]
for any $c\in\mathbf C_T$.

By the relative entropy inequality, for every nonnegative random variable $F$ and every $\lambda>0$,
\begin{equation}\label{rel ent}
E^{P^c}[F]
\leq
\frac1\lambda
\log E^P[\e^{\lambda F}]
+
\frac1\lambda H(P^c\,|\,P),
\end{equation}
where $H(P^c\,|\,P)$ denotes the relative entropy of $P^c$ with respect to
$P$, that is,
\[
H(P^c\,|\,P)
=
E^{P^c}\!\bigg[\log\!\left(\frac{dP^c}{dP}\right)\bigg] \qquad\mbox{whenever $P^c\ll P$, and $H(P^c\,|\,P)=+\infty$ otherwise.}
\]
Now by Girsanov's formula,
\begin{equation}\label{Girsanov}
H(P^c\,|\,P)\big|_{[0,T]}
=
\frac12
E^{P^c}
\bigg[
\int_0^T \|c(s)\|_{a(X_s)}^2\,\d s
\bigg]
\leq C_T<\infty,
\end{equation}
because $c\in\mathbf C_T$ and $a$ is bounded above by \eqref{eq:ellip-bounds} in Assumption \ref{f1}.

In \eqref{rel ent}, choose
\[
F=\e^{-\tau^\omega}.
\]
Since $\tau^\omega=\infty$ $P$-almost surely, we have $F=0$ $P$-almost surely. Hence
\[
E^P[\e^{\lambda F}]=1.
\]
Therefore, for every $\lambda>0$,
\[
E^{P^c}[\e^{-\tau^\omega}]
\leq
\frac1\lambda H(P^c\,|\,P)
\leq
\frac{C_T}{\lambda}.
\]
Letting $\lambda\to\infty$ gives
\[
E^{P^c}[\e^{-\tau^\omega}]=0.
\]
Since $\e^{-\tau^\omega}>0$ on $\{\tau^\omega<\infty\}$, we conclude that
\[
P^c(\tau^\omega<\infty)=0.
\]
Thus $\tau^\omega=\infty$ $P_x^{c,\omega}$-almost surely for every admissible control $c\in\mathbf C_T$.

The existence and uniqueness of the strong solution follow from the global Lipschitz continuity of $a$ and the local Lipschitz continuity of $\sigma$ and $\mathrm{div}\,a$ away from the boundary, together with the fact that the boundary is not attained. This completes the proof of Theorem \ref{theorem 1}.
\qed

\begin{remark}\label{remark:weighted}
The condition $c\in \mathbf C_T$ is stronger than what is needed for the
Girsanov argument above. Indeed, since $a$ is uniformly bounded above,
\[
\|c(s)\|_{a(X_s)}^2
=
\langle a(X_s)c(s),c(s)\rangle
\leq C |c(s)|^2.
\]
Thus
\[
E^P\left[\int_0^T |c(s)|^2\,ds\right]<\infty
\]
implies the intrinsic weighted integrability condition
\[
E^P\left[\int_0^T \|c(s)\|_{a(X_s)}^2\,ds\right]<\infty.
\]
The latter is the quantity that appears in the relative entropy estimate
used in the proof of Theorem~\ref{theorem 1}.\qed 
\end{remark}

\begin{remark}[{\bf Sharpness of the Lipschitz threshold}]
\label{remark:sharpness}

The following model example illustrates that the Lipschitz regularity of $x\mapsto a(x,\omega)$ in Assumption~\ref{f1}(b) is natural from the point of view of boundary non-attainment. Consider the half-plane
\[
D=\{(x,y)\in\R^2:x>0\}
\quad\mbox{with distance-to-the-boundary function}\quad 
d(x,y)=x.
\]
For $0<c<1$, consider the divergence-form operator
\[
L=\nabla\cdot(d^c\nabla)
=
\nabla\cdot(x^c\nabla).
\]
Restricting to functions depending only on the first coordinate reduces the problem to the one-dimensional diffusion with generator
\[
Lf(x)
=
(x^c f'(x))'
=
x^c f''(x)+c x^{c-1}f'(x),
\qquad x>0.
\]
The harmonic equation $\nabla\cdot(x^c\nabla h)=0$
admits nonconstant bounded solutions of the form
\[
h(x)=A+B x^{1-c}
\]
near the boundary $x\to 0$ precisely when $c<1$. Thus, with $\tau_{a,b}= \inf\{t\geq 0: X_t\in \{a,b\}\}$, the stopped process $h(X_{t\wedge\tau_{a,b}})$ is a bounded martingale so that by optional stopping we have $h(x)=E_x[h(X_{\tau_{a,b}}]$. Choosing $h(x)=x^{1-c}$, we get 
$$
P_x[\tau_a<\tau_b]= \frac{b^{1-c}- x^{1-c}}{b^{1-c}-a^{1-c}}
$$
Letting $a\downarrow 0$ and then $b\to\infty$ we get $P_x[\tau_0<\infty]=1$ and thus the process hits the boundary in finite time almost surely.\qed
\end{remark}

\subsection{Example of degenerate mobility satisfying \ref{f1}}\label{sec exmaple a}
In this section we will construct a function $x\mapsto a(x,\omega)$ satisfying the assumptions in \ref{f1} for any fixed $\omega\in\Omega_0$. 
\begin{lemma}\label{lemma}
Let $U\subsetneq \mathbb R^d$ be a nonempty open set, and let
\[
d_U(x):=\operatorname{dist}(x,\partial U),\qquad x\in U.
\]
Then there exists a construction 
\[
U \mapsto \rho_U \qquad\mbox{with} \qquad \rho_U\in C^\infty(U),\qquad \rho>0 \text{ in }U,
\]
such that for every $x\in U$, we have
\begin{equation}\label{rho est}
c\, d_U(x)\le \rho_U(x)\le C\, d_U(x), \qquad\text{and}\qquad 
|D^\alpha \rho_U|\leq C_{|\alpha|} d_U(x)^{1-|\alpha|} 
 \text{ for }\alpha \in \mathbb N_0^d,  
\end{equation}
where the constant  $ C_{|\alpha|}$  depends only on $d$ and 
 $|\alpha|=\sum_1^d \alpha_i$. In particular, $|\nabla \rho_U(x)|\le C_1$ and $\rho_U$ is globally Lipschitz on $U$, and its extension by $0$ on $\mathbb R^d\setminus U$ is globally Lipschitz on $\mathbb R^d$.

Moreover, the construction is translation-covariant in the sense 
\begin{equation}\label{rho tran}
\rho_{U+z}(x+z)=\rho_U(x) \qquad\forall x\in U.
\end{equation}
\end{lemma} 
\begin{proof}
We fix a non-negative function $\eta\in C^\infty_c(B_1(0))$ such that $\eta\equiv 1$ 
on $B_{1/2}(0)$. Define 
$$
r(y)= r_U(y)=  \frac{d_U(y)}{8}, \qquad K_U(x,y)= r(y)^{-d} \eta \bigg(\frac{x-y}{r(y)}\bigg)\geq 0,
$$
and define $\rho_U(\cdot)$ to be the weighted average of $d_U(\cdot)$ with respect to the kernel $K_U$:
$$
\rho_U(x):=\frac{N_U(x)}{S_U(x)}, \quad S_U(x):=\int_{\R^d} K_U(x,y)\, \d y, \quad N_U(x):=\int_{\R^d} d_U(y)K_U(x,y)\,\d y.
$$
We claim that $S_U$ is bounded above and below by positive constants (and thus $\rho_U$ is well-defined).  Indeed, if $K_U(x,y)\neq 0$, then
$|x-y|\le r(y)=\tfrac18 d_U(y)$. Since $d_U$ is $1$-Lipschitz, this implies
$d_U(x)\le \tfrac98 d_U(y)$ and
$d_U(y)\le \tfrac87 d_U(x)$. Hence $r(y)\asymp d_U(x)$ on the support of $K_U(x,\cdot)$.
Now for the upper bound, the support of $K_U(x,\cdot)$ is contained in a ball of radius $C d_U(x)$, and
$K_U(x,y)\le C d_U(x)^{-d}$. Thus
\[
S_U(x)\le C d_U(x)^{-d}\cdot d_U(x)^d \le C.
\]
For the lower bound, if $|x-y|\le c d_U(x)$ with $c$ small enough, then $d_U(y)\asymp d_U(x)$ and
$|\frac{x-y}{r(y)}|\le \tfrac12$,
so $\eta=1$. Hence $K_U(x,y)\ge c d_U(x)^{-d}$
on $B(x,c d_U(x))$, and therefore
\[
S_U(x)\ge c d_U(x)^{-d}\cdot d_U(x)^d \ge c.
\]
Thus $0<c\le S_U(x)\le C$. Next, since $\rho_U(x)$ is a weighted average of $d_U(y)$ and $d_U(y)\asymp d_U(x)$ on the support of $K_U(\cdot,y)$, we obtain
\[
c\,d_U(x)\le \rho_U(x)\le C\,d_U(x).
\]

We now differentiate and note that since $r(y)$ does not depend on $x$,
\[
D_x^\alpha K_U(x,y)
=
r(y)^{-d-|\alpha|}(D^\alpha\eta)\!\left(\frac{x-y}{r(y)}\right).
\]
Because $\eta\in C_c^\infty$, $|D^\alpha\eta|\le C_\alpha$, hence
\[
|D_x^\alpha K_U(x,y)|\le C_\alpha r(y)^{-d-|\alpha|}.
\]
Using again $r(y)\asymp d_U(x)$ on the support and the size of the integration region,
\[
|D^\alpha S_U(x)|\le C_\alpha d_U(x)^{-|\alpha|},
\quad
|D^\alpha N_U(x)|\le C_\alpha d_U(x)^{1-|\alpha|}.
\]
Since $0<c\leq S_U(\cdot)\leq C <\infty$, the quotient rule gives $|D^\alpha (1/S_U)|\leq C_\alpha d_U(x)^{-|\alpha|}$. Since $\rho_U= N_U S_U^{-1}$, $D^\alpha\rho_U= \sum_{\beta \leq \alpha} \binom{\alpha}{\beta}\, D^\beta N_U \,\, D^{\alpha-\beta}(S_U^{-1})$. Therefore, 
\[
|D^\alpha \rho_U(x)|\le C_\alpha d_U(x)^{1-|\alpha|}.
\]
In particular, $|\nabla \rho_U(x)|\le C$. This proves \eqref{rho est}. To prove \eqref{rho tran}
we note that for $z\in\mathbb R^d$,
$d_{U+z}(x+z)=d_U(x)$ and so  $r_{U+z}(y+z)=r_U(y)$ and 
$K_{U+z}(x+z,y+z)=K_U(x,y)$. Thus,  
\[
S_{U+z}(x+z)= \int_{\R^d}K_{U+z} (x+z,y) \d y= \int_{\R^d} K_{U+z} (x+z,y+z) dy= S_U(x).
\]
Likewise $N_{U+z}(x+z)=N_U(x)$, hence
$\rho_{U+z}(x+z)=\rho_U(x)$. This shows \eqref{rho tran} and completes the proof of the lemma. 
\end{proof}

 We are now ready to complete the construction of a diffusion matrix $a$ that satisfies 
the smoothness and positivity conditions assumed in \ref{f1}. 

\begin{cor}\label{cor example a}
Let $\mathcal C(\omega):=\bigcup_{y\in\omega} B_{1/2}(y)$ 
and as before assume $\mathcal C(\omega)$ has a unique unbounded connected component $\mathcal C_\infty(\omega)$. Let
$$
U=\mathcal C_\infty(\omega),
\qquad
\rho(x,\omega):=\rho_{\mathcal C_\infty(\omega)}(x).
$$
with $\rho$ being the function constructed in Lemma \ref{lemma}. 
Let $\theta\in C^\infty([0,\infty))$ satisfy
$\theta(0)=0$, $\theta(r)=r \quad \text{for }0\le r\le \frac12$, 
$\theta(r)=1 \quad \text{for }r\ge 1$, 
$0\le \theta(r)\le 1$, 
$\theta(r)>0 \text{ for } r>0$ and 
$|\theta'(r)|+|\theta''(r)|\le C$.
Define
$$
a(x,\omega):=
\begin{cases}
\theta(\rho(x,\omega)) I_{d\times d},
& x\in \mathcal C_\infty(\omega),\\
0,
& x\notin \mathcal C_\infty(\omega).
\end{cases}
$$
Then the following properties hold:
\begin{itemize}
\item For every $x\in \R^d$ and $\omega\in \Omega$, $a(x,\omega)=a(0,\tau_x \omega)$.
\item The function $\xi(x,\omega):=\theta(\rho(x,\omega))$  satisfies the following:  
for $x\in \mathcal C_\infty(\omega)$, $\xi(x,\omega)\in(0,1]$ and for any $x\in \R^d$ and $\omega\in \Omega$, $\langle a(x,\omega)v,v\rangle
=\xi(x,\omega)|v|^2$ and $\xi(x,\omega)=\xi(0,\tau_x\omega)$. 
\item $\R^d\ni x \mapsto a(x,\omega)$ is globally Lipschitz, 
\item The map 
$$
\mathcal C_\infty(\omega)\ni x \mapsto \sigma(x,\omega):=\sqrt{\theta(\rho(x,\omega))}\, I_{d\times d}
$$ 
is locally Lipschitz on every set $\{x\in \mathcal C_\infty(\omega): d(x,\partial \mathcal C_\infty(\omega))\ge \varepsilon\}$, 
\item The map $x \mapsto \operatorname{div}a(x,\omega)=\nabla\rho(x,\omega)$ 
is bounded and locally Lipschitz away from the boundary, because 
\[
|D^2\rho(x,\omega)|\le C\, d(x,\partial \mathcal C_\infty(\omega))^{-1}.
\]
So $\nabla\rho$ is locally Lipschitz on every region at positive distance from $\partial \mathcal C_\infty(\omega)$.
\end{itemize} 
\end{cor}
\begin{proof} 
The stationarity follows easily from the translation covariance of $\rho$. Indeed, recall the action of 
the translation group $\tau_x\omega=\omega-x$. Therefore,
$\mathcal C_\infty(\tau_x\omega)
=\mathcal C_\infty(\omega)-x$. 
By translation covariance of $\rho$ from Lemma \ref{lemma} applied with $z=-x$,
$$\rho_{\mathcal C_\infty(\omega)-x}(0)
=
\rho_{\mathcal C_\infty(\omega)}(x), \quad\mbox{so}\quad \rho(0,\tau_x\omega)=\rho(x,\omega) \quad\mbox{and}\quad a(0,\tau_x\omega)=a(x,\omega). 
$$
This proves the stationarity. The second property is obvious from teh definition of the ellipticity $\xi$.

Next we show that map $x\mapsto a(x,\omega)$ is globally Lipschitz on $\mathbb R^d$.
Indeed, inside $\mathcal C_\infty(\omega)$,
$\nabla a(x,\omega)
=\theta'(\rho(x,\omega))\nabla\rho(x,\omega) I_{d\times d}$ 
which is bounded because $\theta'$ and $\nabla\rho$ are bounded.
Across the boundary, the extension by zero is Lipschitz. Indeed, for $x\in\mathcal C_\infty(\omega)$,
$|a(x,\omega)|
\le C \rho(x,\omega)
\le C\,\operatorname{dist}(x,\partial\mathcal C_\infty(\omega))$. 
Thus $a(x,\omega)\to 0$ at the boundary at least linearly, and the zero extension is globally Lipschitz.

Next we estbalish the local Lipschitz continuity of the square root 
$\sigma(x,\omega):=\sqrt{\theta(\rho(x,\omega))}\,I_{d\times d}$
for $x\in\mathcal C_\infty(\omega)$ on every set
$\left\{
x\in\mathcal C_\infty(\omega):
\operatorname{dist}(x,\partial\mathcal C_\infty(\omega))\ge \varepsilon
\right\}$. Indeed, by \eqref{rho est}, on such a set,
$\rho(x,\omega)\ge c\varepsilon>0$, and the map $r\mapsto \sqrt{\theta(r)}$ is smooth on intervals bounded away from $0$. 

Finally we show boundedness and local Lipschitz continuity of $\operatorname{div}a$. 
Inside $\mathcal C_\infty(\omega)$, $\operatorname{div}a(x,\omega)=
\theta'(\rho(x,\omega))\nabla\rho(x,\omega)$. This vector field is bounded because $\theta'$ and $\nabla\rho$ are bounded.

Moreover, it is locally Lipschitz away from the boundary. Indeed,
$$
D(\operatorname{div}a)
=
\theta''(\rho)\nabla\rho\otimes\nabla\rho
+
\theta'(\rho)D^2\rho.
$$
On every region at positive distance from $\partial\mathcal C_\infty(\omega)$,
$$
|D^2\rho(x,\omega)|
\le
C\,\operatorname{dist}(x,\partial\mathcal C_\infty(\omega))^{-1}
$$
is bounded. Therefore $\operatorname{div}a(\cdot,\omega)$ is locally Lipschitz away from the boundary.
\end{proof}

\section{Proof of Theorem \ref{theorem 2}}\label{sec proof thm2}

We now prove that, for every $\eps>0$, $T>0$, and $\omega\in\Omega_0$, the variational representation \eqref{eq:u-eps-def} defines a viscosity solution of the Hamilton--Jacobi--Bellman equation
\begin{equation}\label{eq-HJB}
\begin{cases}
\partial_t u_\eps
=
\dfrac{\eps}{2}\mathrm{div}\!\left(
a\!\left(\dfrac{x}{\eps},\omega\right)\nabla u_\eps
\right)
+
H\!\left(
\dfrac{x}{\eps},
\nabla u_\eps,
\omega
\right)
&
\text{in }
(0,T)\times \eps\mathcal C_\infty(\omega),
\\[0.3em]
u_\eps(0,x,\omega)=f(x)
&
\text{on }
\eps\mathcal C_\infty(\omega),
\end{cases}
\end{equation}
in the viscosity sense under Assumptions \ref{f1} and \ref{f2}--\ref{f4}. Observe that, in contrast to the classical setting of PDEs on bounded domains, no boundary condition is imposed on the lateral boundary
\[
(0,T)\times \eps\partial\mathcal C_\infty(\omega).
\]
This is a consequence of Theorem \ref{theorem 1}, which ensures that the controlled diffusion never reaches the boundary.

\medskip

First notice that, by Assumption \ref{f1}, the Hamilton--Jacobi--Bellman equation \eqref{eq-HJB} with $\eps=1$ can be written as
\begin{equation}\label{eq:HJB-general}
\begin{cases}
\partial_t u
=
\mathcal H(x,\nabla u,\mathrm{Hess}_x u)
&
\text{in }V,
\\
u(0,x)=f(x)
&
\text{on }\mathcal C_\infty(\omega),
\end{cases}
\end{equation}
where
\begin{equation}\label{V-calH}
\begin{aligned}
V
&:=
(0,T)\times\mathcal C_\infty(\omega),
\\
\mathcal H(x,p,A)
&:=
\frac12\mathrm{div}(a(x,\omega))\cdot p
+
\frac12\mathrm{Trace}(a(x,\omega)A)
+
H(x,p,\omega).
\end{aligned}
\end{equation}

We now recall the notion of viscosity solutions and introduce some notation. We follow the setup from \cite[Sec.~3.1 and Sec.~7.1]{T13}, which allows for degenerate diffusion matrices $\sigma$ provided they are locally Lipschitz in the spatial variable; see Assumption \ref{f1}(b).

\medskip

For a set $V\subset \R_+\times\R^d$, denote by
$\mathrm{USC}(\overline V)$ the set of upper semicontinuous functions
$w:\overline V\to\R\cup\{\infty\}$. Similarly,
$\mathrm{LSC}(\overline V)$ denotes the set of lower semicontinuous functions.

\begin{definition}
Let $V\subset \R_+\times\R^d$ and let
$u:\overline V\to\R$ be locally bounded.

The upper semicontinuous envelope of $u$ is defined by
\begin{equation}\label{eq:semic-envel-def}
u^*(x)
:=
\inf\big\{
w(x):
w\in \mathrm{USC}(\overline V)
\text{ and }w\geq u
\big\}.
\end{equation}
The lower semicontinuous envelope is defined by
\[
u_*:=-(-u)^*.
\]
\end{definition}

It follows immediately from the definition that
\[
u_*\leq u\leq u^*,
\]
with
\[
u_*\in\mathrm{LSC}(\overline V),
\qquad
u^*\in\mathrm{USC}(\overline V).
\]

\begin{definition}[Viscosity sub/supersolutions]
\label{def:SubSuperSol}
Let $V\subset \R_+\times\R^d$.

\begin{itemize}

\item
A locally bounded function
$u:\overline V\to\R$
is called a viscosity subsolution of
\begin{equation}\label{eq-visc-sub}
\partial_t u
=
\mathcal H(x,\nabla u,\mathrm{Hess}_x u)
\qquad
\text{in }V
\end{equation}
if, for every $(t_0,x_0)\in V$ and every smooth test function $\phi$ defined in a neighborhood of $(t_0,x_0)$ such that
\[
u^*-\phi
\]
has a local maximum at $(t_0,x_0)$, one has
\[
\partial_t\phi(t_0,x_0)
-
\mathcal H(
x_0,
\nabla\phi(t_0,x_0),
\mathrm{Hess}_x\phi(t_0,x_0)
)
\leq0.
\]

\item
A locally bounded function
$u:\overline V\to\R$
is called a viscosity supersolution of \eqref{eq-visc-sub} if, for every $(t_0,x_0)\in V$ and every smooth test function $\phi$ defined in a neighborhood of $(t_0,x_0)$ such that
\[
u_*-\phi
\]
has a local minimum at $(t_0,x_0)$, one has
\[
\partial_t\phi(t_0,x_0)
-
\mathcal H(
x_0,
\nabla\phi(t_0,x_0),
\mathrm{Hess}_x\phi(t_0,x_0)
)
\geq0.
\]

\item
A function is called a viscosity solution if it is both a viscosity subsolution and a viscosity supersolution.

\end{itemize}
\end{definition}

\medskip

\noindent{\bf Proof of Theorem \ref{theorem 2}.}
Fix $T>0$ and $\omega\in\Omega_0$, and set
\[
\mathbf S
:=
(0,T)\times\mathcal C_\infty(\omega). \qquad\mbox{and assume that the admissible controls take values in $U:=\R^d$.}
\]

We consider the controlled diffusion
\begin{equation}\label{eq:controlled-sde-sol}
\d X_t
=
\mathbf b(X_t,c_t)\,\d t
+
\sigma(X_t)\,\d B_t,
\quad \mbox{where}
\quad
\mathbf b(x,c)
:=
a(x)c+\mathrm{div}\,a(x).
\end{equation}
Since $\omega\in\Omega_0$ remains fixed throughout the proof, we suppress the dependence on $\omega$ in the notation.

By Theorem \ref{theorem 1}, for every admissible control
$c\in\mathbf C_T$ and every initial condition
$x\in\mathcal C_\infty(\omega)$, the corresponding diffusion remains in
$\mathcal C_\infty(\omega)$ for all times almost surely. Consequently, the stochastic control problem is well-defined on the random domain without imposing any boundary condition on
$\partial\mathcal C_\infty(\omega)$.

A special class of controls consists of Markovian controls of the form
\begin{equation}\label{Markovian control}
c(s)=\tilde c(s,X_s)
\end{equation}
for a measurable map
\[
\tilde c:[0,T]\times\mathcal C_\infty(\omega)\to\R^d.
\]

The cost functional is defined by
\begin{equation}\label{def cost}
\tilde J_f(t,x,c)
:=
E^P\bigg[
f(X_T^{t,x,c})
-
\int_t^T
L(X_s^{t,x,c},c_s)\,\d s
\bigg],
\end{equation}
where $(X_s^{t,x,c})_{s\geq t}$ denotes the solution of \eqref{eq:controlled-sde-sol} with initial condition
\[
X_t^{t,x,c}=x.
\]

The associated value function is
\begin{equation}\label{def value}
V_f(t,x)
:=
\sup_{c\in\mathbf C_T}
\tilde J_f(t,x,c).
\end{equation}

The corresponding second-order Hamiltonian is
\[
\begin{aligned}
\mathcal H(x,p,A)
&=
\sup_{u\in\R^d}
\bigg[
\mathbf b(x,u)\cdot p
-
L(x,u)
+
\frac12\mathrm{Trace}(a(x)A)
\bigg]
\\
&=
\frac12\mathrm{Trace}(a(x)A)
+
\mathrm{div}(a(x))\cdot p
+
H(x,p).
\end{aligned}
\]
Indeed,
\[
\mathbf b(x,u)\cdot p
=
\langle u,p\rangle_a
+
\mathrm{div}(a(x))\cdot p,
\]
and $H$ and $L$ are related by the Legendre transform with respect to the weighted pairing
$\langle\cdot,\cdot\rangle_a$.

The associated backward Hamilton--Jacobi equation is
\begin{equation}\label{backward}
\partial_tV
=
-\mathcal H(
x,
\nabla_xV,
\mathrm{Hess}_xV
).
\end{equation}

We now verify the assumptions required for \cite[Theorem 7.4]{T13}. By Assumptions \ref{f1}, \ref{f2}-\ref{f3},
\[
a,\quad
\sigma,\quad
\mathrm{div}\,a,
\quad\text{and}\quad
H
\]
have the required continuity and local Lipschitz regularity properties. Moreover, Theorem \ref{theorem 1} guarantees existence and uniqueness of strong solutions to the controlled SDE together with non-attainment of the boundary. It therefore remains only to verify the local boundedness of the value function.

\begin{lemma}\label{le:V.local.bdd}
Under Assumptions \ref{f1} and \ref{f2}--\ref{f4}, for every $T>0$ there exists a constant $\bar C_{T,f}$ such that
\begin{equation}\label{eq:V-locally-bounded}
|V_f(t,x)|
\leq
\bar C_{T,f}(1+|x|)
\qquad
\text{for all }
(t,x)\in[0,T]\times\mathcal C_\infty(\omega).
\end{equation}

Moreover, if $f$ is uniformly bounded, then there exists
$\bar C_{T,f}>0$ such that
\begin{equation}\label{eq:V-unif-bound}
|V_f(t,x)|
\leq
\bar C_{T,f}
\qquad
\text{for all }
(t,x)\in[0,T]\times\mathcal C_\infty(\omega).
\end{equation}
\end{lemma}

\begin{proof}
From Assumption \ref{f4}, we obtain
\begin{equation}\label{eq:f.bound}
|f(y)|
\leq
C_f^{(0)}+C_f^{(1)}|y|,
\qquad
y\in\R^d,
\end{equation}
where
\[
C_f^{(0)}:=1+|f(0)|,
\qquad
C_f^{(1)}:=K_1.
\]
If $f$ is uniformly bounded, we may instead take
\[
C_f^{(0)}=K_f,
\qquad
C_f^{(1)}=0.
\]

For the controlled process,
\[
\begin{aligned}
E^P[|X_s^{t,x,c}|]
\leq
|x|
+
\|\mathrm{div}\,a\|_\infty(s-t)
+
E^P\bigg[
\int_t^s
|a(X_r^{t,x,c})c_r|
\,\d r
\bigg].
\end{aligned}
\]

We first derive a lower bound by choosing $c\equiv0$. Using Assumption \ref{f2},
\[
L(x,0)\leq c_{13},
\]
and therefore
\[
\begin{aligned}
V_f(t,x)
\geq
\tilde J_f(t,x,0)
&=
E^P\bigg[
f(X_T^{t,x,0})
-
\int_t^TL(X_s^{t,x,0},0)\,\d s
\bigg]
\\
&\geq
-
C_f^{(0)}
-
C_f^{(1)}E^P[|X_T^{t,x,0}|]
-
c_{13}T
\geq
-
C_f^{(0)}
-
C_f^{(1)}|x|
-
C\,T.
\end{aligned}
\]

We next derive the upper bound. By Assumption \ref{f2},
\[
L(x,c)
\geq
c_{10}|c|_a^{\alpha'}
-c_{11},
\]
where $\alpha'>1$. Hence
\[
\begin{aligned}
f(X_T^{t,x,c})
-
\int_t^T
L(X_s^{t,x,c},c_s)\,\d s
\leq
C_f^{(0)}
+
C_f^{(1)}|X_T^{t,x,c}|
+
c_{11}T
-
c_{10}\int_t^T|c_s|_a^{\alpha'}\,\d s.
\end{aligned}
\]

Taking expectations and using the estimate
\[
|a(x)c|
\leq
\|a\|_\infty^{1/2}|c|_a,
\]
which follows from
\[
|a(x)c|^2
=
\langle a(x)c,c\rangle_a
\leq
\|a\|_\infty |c|_a^2,
\]
we obtain
\[
\begin{aligned}
\tilde J_f(t,x,c)
\leq
C_f^{(0)}
+
C_f^{(1)}|x|
+
C\,T
+
E^P\bigg[
\int_t^T
\big(
k_*|c_s|_a
-
c_{10}|c_s|_a^{\alpha'}
\big)
\,\d s
\bigg],
\end{aligned}
\]
where $k_*^2:=\|a\|_\infty$. Since $\alpha'>1$,
\[
k_*\gamma-c_{10}\gamma^{\alpha'}
\leq
\bar c
\qquad
\text{for all }\gamma\geq0,
\]
for a suitable constant $\bar c$. Therefore
\[
V_f(t,x)
\leq
C_f^{(0)}
+
C_f^{(1)}|x|
+
C\,T.
\]
This proves \eqref{eq:V-locally-bounded}. If $f$ is uniformly bounded, the same argument yields \eqref{eq:V-unif-bound}.
\end{proof}

We now conclude the proof of Theorem \ref{theorem 2}.

\begin{proof}[Proof of Theorem \ref{theorem 2}]
\begin{enumerate}

\item
By Lemma \ref{le:V.local.bdd}, the value function is locally bounded. Hence all assumptions of \cite[Theorem 7.4]{T13} are satisfied. Therefore $V_f$ is a viscosity solution of the backward equation
\begin{equation}\label{eq:backward-HJB}
\begin{cases}
\partial_tV
=
-\dfrac12
\mathrm{Trace}(a\,\mathrm{Hess}_xV)
-
\mathrm{div}(a)\cdot\nabla V
-
H(x,\nabla V)
&
\text{in }
(0,T)\times\mathcal C_\infty(\omega),
\\[0.3em]
V(T,x)=f(x)
&
\text{on }
\mathcal C_\infty(\omega).
\end{cases}
\end{equation}

Defining
\[
u(t,x):=V(T-t,x)
\]
transforms \eqref{eq:backward-HJB} into \eqref{eq-HJB} for $\eps=1$.

\item
The statement for general $\eps>0$ follows by  $f$ by 
$f_\eps(x):=\frac1\eps f(\eps x)$, and observing that Assumption \eqref{eq:H6} is invariant under this scaling.

\item
If $f$ is uniformly bounded, then by Lemma \ref{le:V.local.bdd},
$|V_{f_\eps}(t,x)|
\leq
K_{f_\eps}+CT$,
where $K_{f_\eps}
=
\frac1\eps K_f$. Hence
\[
\begin{aligned}
|u_\eps(t,x)|
=
\eps
\bigg|
V_{f_\eps}
\bigg(
\frac t\eps,
\frac x\eps
\bigg)
\bigg|
&\leq
\eps
\bigg(
\frac{K_f}{\eps}
+
C\frac T\eps
\bigg)
=
K_f+CT.
\end{aligned}
\]
Thus $u_\eps$ is uniformly bounded in $\eps>0$, 
$t\in[0,T]$ and $x\in\eps\mathcal C_\infty(\omega)$.
\end{enumerate}
\end{proof}

\section{Proof of Theorem \ref{theorem 3}}
\label{sec proof thm3}

\subsection{Verification for the Boolean model}

\noindent{\bf Verification of \ref{assump:est-erg}--\ref{assump:fkg}.}
We recall the definitions and notation from Section \ref{sec Boolean}. Assumption \ref{assump:est-erg} is a consequence of \cite[Propositions 2.6--2.7]{MR96}. Assumption \ref{assump:intensity} follows from \eqref{Poisson}, while \ref{assump:inf-comp} follows from \cite[Theorems 3.5--3.6]{MR96}. Property \ref{assump:chem-dist} can be found in \cite[Lemma 3.4]{CGY11}, and \ref{assump:fkg} appears in \cite[Theorem 2.2]{MR96}.

To verify \ref{assump:exp-dec-dist-indshift} in dimension $d\geq3$, let $\zeta>\zeta_c$ and choose $L>0$ sufficiently large so that in the slab
$\R^{d-1}\times[0,L]$
there almost surely exists an infinite cluster $\mathfrak c_\infty$. The existence of such a slab percolation regime follows from the fact that the critical intensity for $\R^d$ coincides with the limit of the critical intensities for
$\R^2\times[0,L]^{d-2}$
as $L\to\infty$; see \cite[Theorem 1]{T93}. By uniqueness of the infinite cluster in $\R^d$, the slab cluster $\mathfrak c_\infty$ is almost surely contained in ${\mathcal C_\infty}$.

Let $A_L$ denote the event that at least one point in the set
\[
\{je:j=1,\dots,L\}
\]
belongs to $\mathfrak c_\infty$. Then
\[
\{|\mathfrak v_e|\geq Lt\}\cap\{0\in{\mathcal C_\infty}\}
\subset
\bigcap_{k\leq \lfloor t\rfloor}\tau_{kLe}(A_L^c).
\]
Since the events in the intersection are independent, if
$p_L:=\P(A_L)$, then
\[
\P(|\mathfrak v_e|\geq Lt,\ 0\in{\mathcal C_\infty})
\leq
(1-p_L)^{\lfloor t\rfloor}
\]
for all $t>0$. This proves \ref{assump:exp-dec-dist-indshift}.

\subsubsection{Negative moment of the distance to the boundary}\label{subsec proof P6}

We now verify \ref{assump:dist} for the Boolean model. Let us first summarize the main idea before turning to the formal proof. Set
\[
B(\omega):=\R^d\setminus \mathcal C_\infty(\omega),
\]
which is the complement of the infinite cluster. For $x\in\R^d$ and $\eps>0$, define
\[
p(x,\eps)
=
\P\Bigl(
0<\dist(x,B(\omega))<\eps
\Bigr).
\]
Since
\[
\tau_x B(\omega)=B(\tau_x\omega),
\qquad\mbox{we have}\qquad 
\dist(x,B(\omega))
=
\dist(0,B(\tau_x\omega)).
\]
By stationarity of $\P$, it follows that
$p(x,\eps)=p(0,\eps)$.

The quantity $p(0,\eps)$ can be estimated using the geometry of the Boolean model and the properties of the underlying Poisson point process. The key point is that the probability that the origin lies within distance $\eps$ of the boundary is proportional to $\eps$ for small $\eps$.

Indeed, consider the situation in which there exists a unique Poisson point
$x_j$ with
\[
\frac12-\frac{\eps}{3}<|x_j|<\frac12,
\]
and no other Poisson points near the corresponding boundary point
\[
y_j:=
-\frac{1-2|x_j|}{2|x_j|}\,x_j
\in \partial B_{1/2}(x_j).
\]
Then
\[
|y_j|
=
\frac{1-2|x_j|}{2}
<
\eps,
\]
so the boundary of the occupied set lies within distance $\eps$ of the origin.

The probability of such a configuration is of order $\eps$, while the probability that the origin lies within distance $\eps$ of the boundary is at most proportional to the volume of a shell of thickness $\eps$. Consequently,
\[
p(0,\eps)\asymp \eps
\qquad
\text{as }\eps\downarrow0.
\]
This implies that
\[
\E_0\bigl[\dist(0,\partial\mathcal C_\infty(\omega))^{-\chi}\bigr]
<\infty
\]
if and only if $\chi<1$.

We now turn to the formal proof.

\begin{lemma}
Let $\P=\P^\zeta$ denote the supercritical Boolean model and define, for all $\omega\in\Omega_0$,
\[
B(\omega):=\mathbb R^d\setminus \mathcal C_\infty(\omega),
\qquad
\delta(\omega):=
\dist\bigl(0,B(\omega)\bigr)
=
\dist\bigl(0,\partial\mathcal C_\infty(\omega)\bigr).
\]
Then there exist constants $c,C,\varepsilon_0>0$ such that for all
$0<\varepsilon<\varepsilon_0$,
\[
c\,\varepsilon
\le
\P_0(\delta<\varepsilon)
\le
C\,\varepsilon.
\]
Consequently,
\[
\E_0[\delta^{-\chi}]<\infty
\qquad\Longleftrightarrow\qquad
\chi\in (0,1).
\]
\end{lemma}

\begin{proof}
For $x\in\mathbb R^d$ and $\varepsilon>0$, define
\begin{equation}\label{def p}
p(x,\varepsilon)
:=
\P\Bigl(
0<
\dist(x,B(\omega))
<
\varepsilon
\Bigr).
\end{equation}
Since $B(\tau_x\omega)=\tau_x B(\omega)$,
we have $\dist(x,B(\omega))
=
\dist(0,B(\tau_x\omega))$ and 
by stationarity of $\P$, $p(x,\varepsilon)=p(0,\varepsilon)$.
Moreover,
\[
\{\delta>0\}
=
\{0\in\mathcal C_\infty\}
=
\Omega_0.
\]
Therefore,
\begin{equation}\label{rel}
\begin{aligned}
p(0,\varepsilon)
=
\P\Bigl(
0<
\dist(0,B(\omega))
<
\varepsilon
\Bigr)
&=
\P\bigl(
\Omega_0\cap\{\delta<\varepsilon\}
\bigr)
=
\P(\Omega_0)\,
\P_0(\delta<\varepsilon).
\end{aligned}
\end{equation}
Thus it suffices to prove that there exist constants $c,C>0$ such that
\begin{equation}\label{asymp}
c\varepsilon
\le
p(0,\varepsilon)
\le
C\varepsilon.
\end{equation}

Indeed, using
\[
\E[X^{-\chi}]
=
\chi\int_0^\infty
\varepsilon^{-\chi-1}
\P(X<\varepsilon)
\,\d\varepsilon
\]
for positive random variables $X$, and choosing
$X(\omega)=\delta(\omega)$,
we obtain
$$
\E_0[\delta^{-\chi}]
=
\chi
\int_0^\infty
\varepsilon^{-\chi-1}
\P_0(\delta<\varepsilon)
\,\d\varepsilon.
$$
Then with $\P_0(\delta<\varepsilon)\asymp\varepsilon$ 
near $0$, the integral converges if and only if $\chi<1$.
It therefore remains to prove \eqref{asymp}.

\medskip

\noindent{\bf Upper bound.}
If $0<
\dist(0,B(\omega))
<
\varepsilon$,
then there exists
$y\in\partial\mathcal C_\infty(\omega)$
such that $|y|<\varepsilon$. Since
\[
\mathcal C_\infty(\omega)
\subset
\mathcal C(\omega)
=
\bigcup_{x\in\omega}B_{1/2}(x),
\]
such a boundary point must lie on the boundary of at least one grain
$B_{1/2}(x)$, so
$|y-x|=\frac12$. Therefore,
$\Bigl||x|-\frac12\Bigr|
\leq
|y|
<
\varepsilon$, 
which implies
$x\in
B_{1/2+\varepsilon}(0)
\setminus
B_{1/2-\varepsilon}(0)$.

Hence
\[
\{0<\dist(0,B(\omega))<\varepsilon\}
\subset
\Bigl\{
\omega\cap
\bigl(
B_{1/2+\varepsilon}(0)
\setminus
B_{1/2-\varepsilon}(0)
\bigr)
\neq\emptyset
\Bigr\}.
\]

Since $\omega$ is a Poisson point process,
\[
\begin{aligned}
p(0,\varepsilon)
&\le
1-
\exp\Bigl(
-\zeta
\bigl(
|B_{1/2+\varepsilon}|
-
|B_{1/2-\varepsilon}|
\bigr)
\Bigr)
\\
&\le
C_1\varepsilon,
\end{aligned}
\]
because $|B_{1/2+\varepsilon}|
-
|B_{1/2-\varepsilon}|
=
O(\varepsilon)$.

\medskip

\noindent{\bf Lower bound. Step 1.}
Choose a small open patch
\[
\Gamma\subset S^{d-1}
\]
around the vector $e_1$ such that for all sufficiently small $\varepsilon$,
\begin{equation}\label{Seps}
S_\varepsilon
:=
\Bigl\{
x\in\mathbb R^d:
\frac12-\frac{\varepsilon}{3}<|x|<\frac12,
\quad
\frac{x}{|x|}\in\Gamma
\Bigr\}
\subset V,
\qquad
|S_\varepsilon|
=
c_0\varepsilon+O(\varepsilon^2),
\quad
c_0>0.
\end{equation}

Define the event
\[
L_\varepsilon
:=
\Bigl\{
\#(\omega\cap S_\varepsilon)=1,
\quad
\#\bigl(
\omega\cap(B_{2/3}(0)\setminus S_\varepsilon)
\bigr)=0
\Bigr\}.
\]
and note that by independence of Poisson counts on disjoint sets,
\begin{equation}\label{Leps}
\begin{aligned}
\P(L_\varepsilon)
&=
\P\bigl(
\#(\omega\cap S_\varepsilon)=1
\bigr)
\,
\P\bigl(
\#(
\omega\cap(B_{2/3}(0)\setminus S_\varepsilon)
)=0
\bigr)
=
\zeta |S_\varepsilon|
\e^{-\zeta |S_\varepsilon|}
\,
\e^{-\zeta |B_{2/3}(0)\setminus S_\varepsilon|}
\\
&=
\zeta |S_\varepsilon|
\e^{-\zeta |B_{2/3}(0)|}
=
c_1\varepsilon+O(\varepsilon^2)
\ge
c_2\varepsilon
\end{aligned}
\end{equation}
for all sufficiently small $\varepsilon$. On $L_\varepsilon$, let $x_j$ denote the unique point in $S_\varepsilon$, and define
\[
y_j
:=
-\frac{1-2|x_j|}{2|x_j|}\,x_j.
\qquad\mbox{so that}\quad 
y_j\in \partial B_{1/2}(x_j), \quad\mbox{and}\quad 
|y_j|
=
\frac{1-2|x_j|}{2}
<
\varepsilon.
\]
Moreover, for $\varepsilon<1/6$,
$B_{1/2}(y_j)\subset B_{2/3}(0)$.
Since there are no Poisson points in
$B_{2/3}(0)\setminus S_\varepsilon$,
and $x_j$ is the unique point in $S_\varepsilon$, it follows that no Poisson point other than $x_j$ lies within distance $1/2$ of $y_j$. Hence
$$y_j\in\partial\mathcal C(\omega).
$$
Finally, since $|x_j|<\frac12$,
we have
\[
0\in B_{1/2}(x_j).
\]

\medskip

\noindent{\bf Step 2.}
Choose an open ball
\[
U_1\subset \mathbb R^d\setminus B_{2/3}(0)
\quad\mbox{such that}\quad 
B_{1/2}(x)\cap B_{1/2}(u)\neq\emptyset
\quad
\forall
x\in S_\varepsilon,
\quad
u\in U_1.
\]
This is possible because $S_\varepsilon$ is contained in a fixed small sector.
Next choose finitely many pairwise disjoint open balls
\[
U_1,\dots,U_N
\subset
\R^d\setminus B_{2/3}(0)
\]
and a distant box $Q'$ such that for every $k=1,\dots,N-1$,
\[
|u_k-u_{k+1}|<1
\quad\mbox{for all}\quad 
u_k\in U_k,
\quad
u_{k+1}\in U_{k+1},
\quad\mbox{and such that}\quad 
B_{1/2}(u_N)\cap Q'\neq\emptyset
\quad \forall u_N\in U_N.
\]
If
\begin{equation}\label{J}
J
:=
\bigcap_{k=1}^N
\{\omega\cap U_k\neq\emptyset\}, \quad\mbox{then}\quad \P(J)
=
\prod_{k=1}^N
\bigl(
1-\e^{-\zeta |U_k|}
\bigr)
>0.
\end{equation}
since the sets $U_k$ are pairwise disjoint. On the event $L_\varepsilon\cap J$, the ball
$B_{1/2}(x_j)$ is connected by a finite chain of overlapping occupied balls to the box $Q'$.

\medskip

\noindent{\bf Step 3.}
Let $D$ be a bounded region containing the origin together with all the sets
\[
U_1,\dots,U_N.
\]
Choose the box
$Q'\subset D^c$
sufficiently large so that the event
\begin{equation}\label{HD}
\begin{aligned}
&H_D
:=
\bigg\{
Q'
\text{ is connected to infinity by an occupied path contained entirely in }D^c
\bigg\} \\
&\qquad\qquad\qquad \mbox{satisfies $\P(H_D)>0$}. 
\end{aligned}
\end{equation}
Such an event exists by supercriticality together with the local modification property of the Boolean model. We note that (i) the event $L_\varepsilon$ depends only on the configuration inside
$B_{2/3}(0)$; (ii)
the event $J$ depends only on the configuration inside
$D\setminus B_{2/3}(0)$;
and (iii) the event $H_D$ depends only on the configuration outside $D$.
Hence these events are independent, and therefore
\[
\P(L_\varepsilon\cap J\cap H_D)
=
\P(L_\varepsilon)\P(J)\P(H_D).
\]
On the event
$L_\varepsilon\cap J\cap H_D$,
the ball $B_{1/2}(x_j)$
is connected to infinity through occupied balls. Hence
$B_{1/2}(x_j)\subset\mathcal C_\infty(\omega)$.
Since $y_j\in\partial B_{1/2}(x_j)$
and no other ball covers $y_j$, it follows that
$y_j\in\partial\mathcal C_\infty(\omega)$.
Because $|y_j|<\varepsilon$,
we conclude that
\[
\delta(\omega)<\varepsilon.
\]
Therefore,
\[
L_\varepsilon\cap J\cap H_D
\subset
\{\delta<\varepsilon\},
\]
and thus, by \eqref{Leps}, \eqref{J} and \eqref{HD}, 
\[
\P(\delta<\varepsilon)
\ge
\P(L_\varepsilon)\P(J)\P(H_D)
\ge
c\varepsilon
\]
for some constant $c>0$ and all sufficiently small $\varepsilon$. This proves \eqref{asymp}.
\end{proof}

\subsection{Verification for the continuum random cluster model}

We now verify that assumptions
\ref{assump:est-erg}--\ref{assump:inf-comp}
hold for the continuum random cluster model in the supercritical regime.
First, by construction, the CRCM is a stationary Gibbs point process. In particular, its law is invariant under translations of $\mathbb R^d$. By \cite[Theorem 1]{DH15}, there exists at least one stationary CRCM$(q,\zeta)$, which can moreover be chosen ergodic. This proves \ref{assump:est-erg}.

Assumption \ref{assump:intensity} follows from the local finite-energy property of the CRCM. More precisely, conditional on the configuration outside a bounded domain, the law inside the domain is absolutely continuous with respect to the Poisson point process with strictly positive density; see \cite[Lemma 7.1]{H18}. In particular, the CRCM has locally finite intensity.

Next, by \cite[Theorem 1]{H18}, any CRCM satisfies
\[
\P(N_\infty\le1)=1,
\]
that is, there exists almost surely at most one unbounded connected component.

Moreover, under the assumptions of \cite[Theorem 1]{H18} (in particular either $q\ge1$ together with an integrability condition on the radii, or $q<1$ with bounded radii), there exists $z_1<\infty$
such that for all
$z>z_1$, one has
\[
\P(N_\infty\ge1)>0.
\]
Since the event
$\{N_\infty=1\}$
is translation invariant, ergodicity implies
$\P(N_\infty=1)=1$
in the supercritical regime.

Finally, stationarity together with the existence of a unique infinite occupied component implies
\[
\P(0\in\mathcal C_\infty(\omega))>0.
\]
This proves assumption \ref{assump:inf-comp}.

\subsubsection{Sketch of ideas for verifying \ref{assump:chem-dist} for the random cluster model}
\label{subsec sketch P4}

We briefly sketch the ideas behind the proof of \ref{assump:chem-dist}. We first introduce some notation.

Given a finite subset
$\Lambda\subset\Z^d$,
its outer boundary is defined by
\[
\partial^{\mathrm{out}}\Lambda
:=
\{
x\in\Lambda^c:
\exists\,y\in\Lambda
\text{ with }
|x-y|_1=1
\}.
\]
If $B=\prod_{i=1}^d[a_i,b_i]$
is a box in $\R^d$, we say that a connected component
$C\subset B$
is {\it crossing for $B$} if for every
$i\in\{1,\dots,d\}$ there exist points
$x(i),y(i)\in C$
such that
\[
|x_i(i)-a_i|\le\frac12,
\qquad
|y_i(i)-b_i|\le\frac12.
\]

For each $M>0$, define a random field
$
\{X_z:z\in\Z^d\}
$
as follows. Let $B_z$ and $B_z^+$ denote concentric cubes centered at $z$ of radius
$M$ and $\frac{5M}{4}$ respectively. Define the event
$A_z$ to be the event that:
\begin{itemize}
\item
there exists a unique crossing component
$C_z$
in $B_z^+$;

\item
every sub-box of radius $M/4$ centered at
$z+h$ for 
$h\in\Z^d\cap[-M/2,M/2]^d$
contains a unique crossing component;

\item
all such crossing components are connected to $C_z$.
\end{itemize}

Setting 
\[
X_z:=\mathbf 1_{A_z}.
\]
if one can prove that for some integer $k\ge1$,
\begin{equation}\label{eq:eq24}
\lim_{M\to\infty}
\sup_{z\in\Z^d}
\mathrm{ess\,sup}\,
\P\Bigl(
X_z=1
\mid
\sigma(X_y:|y-z|_\infty>k)
\Bigr)
=
1,
\end{equation}
then, by \cite[Theorem 1.3]{LSS97}, for every $p\in(0,1)$ and sufficiently large $M=M(p)$, the product measure of i.i.d.\ Bernoulli random variables
\[
\{Y_z:z\in\Z^d\}
\]
with parameter $p$ is stochastically dominated by the law of
\[
\{X_z:z\in\Z^d\}.
\]

One may then follow the renormalization argument from \cite[pp.~160--161]{CGY11} to prove \ref{assump:chem-dist}. To verify \eqref{eq:eq24}, see for example \cite[Theorem 3.1]{P96} for $d\ge3$ and \cite[Theorem 9]{CM04} for $d=2$.

Finally, we mention that, at least for sufficiently large $\zeta$ (large enough to exceed an appropriate slab critical parameter), one expects that \ref{assump:exp-dec-dist-indshift} also holds for the CRCM. Regarding the FKG inequality, we could not find an explicit statement in the literature for the CRCM. However, adapting the arguments from \cite[Theorem 4.17]{G06} and \cite[Theorem 2.2]{MR96}, we expect that the desired inequality should remain valid.

\subsection{Some discrete percolation models with long range correlations}
Beyond the random cluster model, there are many examples on the discrete lattice which exhibit long-range correlations and satisfy assumptions analogous to \ref{assump:est-erg}-\ref{assump:fkg}. Examples include random interlacements, vacant sets of random interlacements, and level sets of the Gaussian free field in $d\geq3$; see \cite{Sz10,T09,T09aap,CP12,DRS14,PRS15}. For these models, we refer to \cite{BMO16}, where assumptions analogous to \ref{assump:est-erg}-\ref{assump:fkg} appear and are verified to some extent in the context of quenched large deviations for simple random walks on percolation clusters. We refer to \cite{KV86,SS04,BB07,MP07,PRS15} for earlier results on invariance principle.

\noindent{\bf Acknowledgement.} The third author would like to thank S.R.S. Varadhan for many insightful discussions. The research of the third author
is supported by Deutsche Forschungsgemeinschaft (DFG, German research
foundation), under Germany's Excellence Strategy EXC 2044-390685587,
Mathematics M\"unster: Dynamics-Geometry-Structure.

\end{document}